\pdfoutput=1 
\documentclass[10pt]{amsart}
\input{header}

\newcommand{\defn}{\emph}

\newcommand{\period}{\rlap{\ .}}
\newcommand{\comma}{\rlap{\ ,}}
\newcommand{\andeq}{\text{\qquad and\qquad}}

\newcommand{\colonequals}{\coloneq}
\newcommand{\coproduct}{\sqcup}
\newcommand{\coproductlimits}{\operatornamewithlimits{\coproduct}}
\newcommand{\timeslimits}{\operatornamewithlimits{\times}}
\renewcommand{\ast}{{\mathbin{*}}}

\newcommand{\isomto}{\similarrightarrow}
\newcommand{\equivto}{\similarrightarrow}
\newcommand{\equivalent}{\simeq}

\newcommand{\wedgesum}{\vee}
\newcommand{\smashprod}{\wedge}

\DeclareMathOperator{\pr}{pr}
\DeclareMathOperator{\op}{op}
\renewcommand{\colim}{\operatornamewithlimits{colim}}
\DeclareMathOperator{\Grp}{Grp}
\DeclareMathOperator{\Pro}{Pro}
\DeclareMathOperator{\Cech}{\check{C}}
\DeclareMathOperator{\trun}{\tau}

\newcommand{\Hopf}{h}

\renewcommand{\E}{\operatorname{E}}

\DeclareMathOperator{\Spec}{Spec}
\DeclareMathOperator{\Lmot}{L_{mot}}
\newcommand{\Eone}{\mathbf{E}_{1}}

\newcommand{\Sph}[1]{\upS^{#1}}

\newcommand{\paren}[1]{\left(#1\right)}

\newcommand{\J}{\mathup{J}}

\newcommand{\uppi}{\mathup{\pi}}
\newcommand{\uprho}{\mathup{\rho}}
\newcommand{\upsigma}{\mathup{\sigma}}

\newcommand{\upC}{\mathup{C}}
\newcommand{\upS}{\mathup{S}}
\newcommand{\upU}{\mathup{U}}

\renewcommand{\GG}{\mathbf{G}}
\newcommand{\GGm}{\GG_{\operatorname{m}}}
\newcommand{\ZZ}{\mathbf{Z}}

\newcommand{\Spc}{\mathbf{Spc}}
\renewcommand{\Set}{\mathbf{Set}}
\newcommand{\Deltaop}{\mathbf{\Delta}^{\op}}

\DeclareMathOperator{\Nis}{nis}
\renewcommand{\Sm}{\mathbf{Sm}}
\newcommand{\Shnis}{\Sh_{\Nis}}

\newcommand{\dF}{\mathcal{F}}
\newcommand{\HH}{\mathbf{H}}
\newcommand{\dI}{\mathcal{I}}

\usepackage{enumerate}
\newcommand{\enumref}[2]{(\hyperref[#1.#2]{\ref*{#1}.\ref*{#1.#2}})}

\defbibheading{references}[\refname]{%
    \section*{#1}%
    \addcontentsline{toc}{part}{References} %
    \markboth{#1}{#1}
    }

\title{On the James and Hilton--Milnor Splittings, \& the metastable EHP sequence}

\author{Sanath Devalapurkar}

\author{Peter Haine}

\date{\today}

\begin{document}

\begin{abstract}
    This note provides modern proofs of some classical results in algebraic topology, such as the James Splitting, the Hilton--Milnor Splitting, and the metastable EHP sequence.
    We prove fundamental splitting results
    \begin{equation*}
        \Sigma \Omega \Sigma X \simeq \Sigma X \vee (X\wedge \Sigma\Omega \Sigma X) \andeq \Omega(X \vee Y) \simeq \Omega X\times \Omega Y\times \Omega \Sigma(\Omega X \wedge \Omega Y)
    \end{equation*}
    in the maximal generality of an $\infty$-category with finite limits and pushouts in which pushouts squares remain pushouts after basechange along an arbitrary morphism (i.e., Mather's Second Cube Lemma holds).
    For connected objects, these imply the classical James and Hilton--Milnor Splittings.
    Moreover, working in this generality shows that the James and Hilton--Milnor splittings hold in many new contexts, for example in: elementary $ \infty $-topoi, profinite spaces, and motivic spaces over arbitrary base schemes.
    The splitting results in this last context extend Wickelgren and Williams' splitting result for motivic spaces over a perfect field.
    We also give two proofs of the metastable EHP sequence in the setting of $ \infty $-topoi: the first is a new, non-computational proof that only utilizes basic connectedness estimates involving the James filtration and the Blakers--Massey Theorem, while the second reduces to the classical computational proof.
\end{abstract}

\maketitle

\tableofcontents


\section{Introduction}

A classical result of James shows that given a pointed connected space $ X $, the homotopy type $ \Sigma\Omega\Sigma X $ given by suspending the loopspace on the suspension of $ X $ splits as a wedge sum
\begin{equation}\label{eq:infJames}
    \Sigma\Omega\Sigma X \equivalent \bigvee_{i\geq 1} \Sigma X^{\wedge i}
\end{equation}
of suspensions of smash powers of $ X $ \cites{MR503007}{MR73181}.
Hilton and Milnor proved a related splitting result \cites{MR68218}{MR69492}[Theorem 3]{Milnor:FK}: given pointed connected spaces $ X $ and $ Y $, they showed that there is a homotopy equivalence
\begin{equation}\label{eq:infHM}
    \Omega\Sigma(X \wedgesum Y) \equivalent \Omega\Sigma X \times \Omega\Sigma Y \times \Omega\Sigma\paren{\bigvee_{i,j \geq 1} X^{\wedge i} \smashprod Y^{\wedge j}} \period
\end{equation}
In the classical setting, these splitting results follow from combining a connectedness argument using the hypothesis that $ X $ and $ Y $ are connected with the following more fundamental splittings: given any pointed spaces $ X $ and $ Y $, there are natural equivalences
\begin{equation}\label{eq:JamesHMfundamental}
    \Sigma \Omega \Sigma X \simeq \Sigma X \vee (X\wedge \Sigma\Omega \Sigma X) \andeq \Omega(X \vee Y) \simeq \Omega X\times \Omega Y\times \Omega \Sigma(\Omega X \wedge \Omega Y) \period
\end{equation}
The first objective of this note is to provide clear, modern, and non-computational proofs of these fundamental splittings \eqref{eq:JamesHMfundamental}. 
The only property particular to the $ \infty $-category of spaces that our proofs utilize is \textit{Mather's Second Cube Lemma} \cite{MR402694}; this asserts that pushout squares remain pushouts after basechange along an arbitrary morphism (see \cref{subsec:universalpushout}).
Hence these `fundamental' James and Hilton--Milnor Splittings hold in any $ \infty $-category where we can make sense of suspensions, loops, wedge sums, and smash products, and have access to Mather's Second Cube Lemma:

\begin{theorem}[Fundamental James Splitting; \Cref{james}]\label{introthm:James}
    Let $\dX$ be an $\infty$-category with finite limits and pushouts, and assume that Mather's Second Cube Lemma holds in $ \dX $.
    Then for every pointed object $ X\in \dX_{\ast} $, there is a natural equivalence
    \begin{align*}
       \Sigma \Omega \Sigma X &\simeq \Sigma X \vee (X\wedge \Sigma\Omega \Sigma X) \period \\ 
        \intertext{Moreover, for each integer $ n \geq 1 $ then there is a natural equivalence}
        \Sigma \Omega \Sigma X &\simeq \paren{\bigvee_{1 \leq i \leq n} \Sigma X^{\wedge i}} \vee (X^{\smashprod n} \wedge \Sigma\Omega \Sigma X) \period
    \end{align*}
\end{theorem}

In general, the infinite splitting \smash{$ \Sigma\Omega\Sigma X \equivalent \bigvee_{i\geq 1} \Sigma X^{\wedge i} $} need not hold; roughly speaking, the problem is that if $ X $ is not connected, then $ (X^{\smashprod n} \wedge \Sigma\Omega \Sigma X) $ need not vanish as $ n \to \infty $.
However, there is always a natural map \smash{$ \bigvee_{i\geq 1} \Sigma X^{\wedge i} \to \Sigma\Omega\Sigma X $}. 

\begin{theorem}[Fundamental Hilton--Milnor Splitting; \Cref{hilton-milnor}]\label{introthm:Hilton-Milnor}
    Let $\dX$ be an $\infty$-category with finite limits and pushouts, and assume that Mather's Second Cube Lemma holds in $ \dX $.
    Then for every pair of pointed objects $X,Y\in \dX_{\ast}$ there is an natural equivalence
    \begin{align*}
        \Omega(X \vee Y) &\simeq \Omega X\times \Omega Y\times \Omega \Sigma(\Omega X \wedge \Omega Y) \period 
    \end{align*}
\end{theorem}

\noindent In \cref{subsec:infsplitting} we explain in what generality the the infinite James Splitting \eqref{eq:infJames} and Hilton--Milnor Splitting \eqref{eq:infHM} hold, and how to deduce them from \Cref{introthm:James,introthm:Hilton-Milnor}.

It might seem that knowing that the James and Hilton--Milnor Splittings in this level of generality is of dubious advantage; the settings in which one is most likely to want to apply these splittings are the $ \infty $-category $ \Spc $ of spaces (where the results are already known), or an $ \infty $-topos (where the results follows immediately from the results for $ \Spc $; see \cref{subsec:infsplitting}).
However, algebraic geometry provides an example that does not immediately follow from the result for spaces: motivic spaces.
The obstruction is that the $ \infty $-category of motivic spaces over a scheme is not an $ \infty $-topos; since motivic localization almost never commutes with taking loops, knowing the James and Hilton--Milnor Splittings in the $ \infty $-topos of Nisnevich sheaves does not allow one to deduce that they hold in motivic spaces.

Wickelgren and Williams used the \textit{James filtration} to prove that the infinite James Splitting \eqref{eq:infJames} holds for $ \AA^1 $-connected motivic spaces over a perfect field \cite[Theorem 1.5]{MR3981318}.
The reason for the restriction on the base is because their proof relies on Morel's unstable $ \AA^1 $-connectivity Theorem \cite[Theorems 5.46 and 6.1]{MR2934577}, which implies that motivic localization commutes with loops \cites[Theorem 2.4.1]{MR3654105}[Theorem 6.46]{MR2934577}.
However, the unstable $ \AA^1 $-connectivity property does not hold for higher-dimensional bases \cites[Remark 3.3.5]{MR3654105}{MR2235615}, so a different method is needed if one wants to prove James and Hilton--Milnor Splittings for motivic spaces over more general bases.
This is where our generalization pays off: work of Hoyois \cite[Proposition 3.15]{MR3570135} shows that, in particular, Mather's Second Cube Lemma holds in motivic spaces over an \textit{arbitrary} base scheme.
Therefore, \Cref{introthm:James,introthm:Hilton-Milnor} apply in this setting.
We use these splittings to give a description of the motivic space $ \Sigma \Omega \PP^1 $ in terms of wedges of motivic spheres \smash{$ \Sph{i+1,i} $} (\Cref{motivic-james}), and also give a new description of \smash{$ \Omega \Sigma(\PP^1 \smallsetminus \{0,1,\infty\}) $} (\Cref{example:P1threepts}).
Over a perfect field, we give a new decomposition of \smash{$ \Omega \Sigma^2(\PP^1 \smallsetminus \{0,1,\infty\}) $} in terms of motivic spheres of the form \smash{$ \Sph{2m+1,m} $} (\Cref{ex:OmegaSigma2P13pts}).

The second goal of this note is to give a modern construction of the metastable EHP sequence in an $ \infty $-topos $ \dX $.
For every pointed connected object $ X \in \dX_{\ast} $, the James Splitting provides \textit{Hopf maps}
\begin{equation*}
    \Hopf_n \colon \Omega \Sigma X\to \Omega \Sigma (X^{\wedge n}) \period
\end{equation*}
There is also a \textit{James filtration} $ \{\J_{m}(X)\}_{m \geq 0} $ on $ \Omega \Sigma X $, and, moreover, the composite
\begin{equation}\label{seq:Jamesnotfib}
    \begin{tikzcd}[sep=1.5em]
        \J_{n-1}(X) \arrow[r] & \Omega \Sigma X \arrow[r, "\Hopf_n"] & \Omega \Sigma X^{\wedge n}
    \end{tikzcd}
\end{equation} 
is trivial.
The sequence \eqref{seq:Jamesnotfib} is not a fiber sequence in general\footnote{When $ \dX $ is the $ \infty $-category of spaces and $ X $ is a sphere, James and Toda proved that, roughly, the sequence \eqref{seq:Jamesnotfib} becomes a fiber sequence after \textit{$ p $-localization}.
See \cites{MR77922}{MR79263}{MR0143217} for a precise statement.}, but is in the \textit{metastable range}:

\begin{theorem}[metastable EHP sequence; \Cref{metastable}]\label{introthm:metastable}
    Let $ \dX $ be an $ \infty $-topos, $ k \geq 0 $ an integer, and $ X \in
    \dX_{\ast} $ a pointed $ k $-connected object.
    Then for each integer $ n \geq 1 $, the morphism $ \J_{n-1}(X) \to \fib(\Hopf_{n})$ is $ ((n+1)(k+1) - 3)$-connected.
\end{theorem}

\noindent We note here that a morphism is $ m $-connected in our terminology if and only if it is $ (m+1) $-connected in the classical terminology (see \Cref{warning:connectedness}).

We provide two proofs of \Cref{introthm:metastable}.
The first proof is new and non-computational; it only makes use of some basic connectedness estimates involving the James filtration and the Blakers--Massey Theorem.
In the second proof we simply note that \Cref{introthm:metastable} for a general $ \infty $-topos follows immediately from the claim for the $ \infty $-topos of spaces.
In the case of spaces, we provide a computational proof; we include this second proof because we were unable to find the computational proof we were familiar with in the literature.


\subsection{Linear overview}

We have written this note with two audiences in mind: the student interested in seeing proofs of \Cref{introthm:James,introthm:Hilton-Milnor,introthm:metastable} in the classical setting of spaces, and the expert homotopy theorist interested in applying these results to more general contexts such as motivic spaces or profinite spaces.
The student can always take $ \dX $ to be the $ \infty $-category of spaces, and the expert can safely skip the background sections provided for the student.
We also note that this text should still be accessible to the reader familiar with homotopy (co)limits but unfamiliar with higher categories, since all we use in our proofs are basic manipulations of homotopy (co)limits.

\Cref{james-section} is dedicated to proving \Cref{introthm:James}.
In \cref{subsec:universalpushout}, we provide background on Mather's Second Cube Lemma and the universality of pushouts.
In \cref{subsec:Jamessplitting}, we provide a proof of the James Splitting. 
Our proof is roughly the same as proofs presented elsewhere \cites{hopkins-course-notes}[\S17.2]{MR2839990}{Wilson:James}, but it seems that the generality of the argument we present here is not very well-known.

\Cref{sec:Hilton-Milnor} provides a quick proof of \Cref{introthm:Hilton-Milnor}.
Again, shadows of the proof we provide appear in the literature \cites{MR281202}[\S 2 \& 3]{MR2428352}[\S17.8]{MR2839990}, but it seems that the generality of the proof has not been completely internalized by the community.
As an application, we use work of Wickelgren \cite[Corollary 3.2]{MR3521594} to give a new description of the motivic space $ \Omega \Sigma(\PP^1 \smallsetminus \{0,1,\infty\}) $ (\Cref{motivic-hilton-milnor}).

\Cref{sec:conncectednessssplit} explains how to use a connectedness argument to prove the infinite James Splitting \eqref{eq:infJames} and Hilton--Milnor Splitting \eqref{eq:infHM} for pointed connected objects of an $ \infty $-topos, and for pointed $ \AA^1 $-connected motivic spaces over a perfect field.
In \cref{subsec:connectedness}, we begin by recalling the basics of connectedness and the Blakers--Massey Theorem in an $ \infty $-topos; this material is also instrumental in our proof of \Cref{introthm:metastable}.
\Cref{subsec:infsplitting} presents the connectedness argument needed to deduce the infinite splittings and defines the Hopf maps appearing in \Cref{introthm:metastable}.
As an application we give a new description of the motivic space \smash{$ \Omega \Sigma^2(\PP^1 \smallsetminus \{0,1,\infty\}) $} over a perfect field (\Cref{ex:OmegaSigma2P13pts}).

\Cref{sec:EHP} is dedicated to proving \Cref{introthm:metastable}.
In \cref{subsec:Jamesconstruction}, we provide the background on the James filtration needed to understand the statement of \Cref{introthm:metastable}, as well as some connectedness estimates we need to prove \Cref{introthm:metastable}.
In \cref{subsec:splitJamesconstr}, we give a refinement of the James Splitting in terms of the James filtration.
In \cref{subsec:EHPproofs}, we first provide a proof of \Cref{introthm:metastable} using the Blakers--Massey Theorem (which we have not seen elsewhere), and then record for posterity what we imagine is the standard computational proof of \Cref{introthm:metastable}.

\begin{acknowledgements}
    We are grateful to Tom Bachmann for pointing out that colimits are universal in motivic spaces.
    We are indebted to André Joyal for noticing some errors in an earlier version as well as alerting us to our misuse of terminology regarding connectedness. 
    We are grateful to Tomer Schlank for explaining why the infinite version of the James Splitting requires $ X $ to be connected, and for explaining some subtleties around \Cref{lem:pr2pushout}.
    We thank Elden Elmanto, Marc Hoyois, and Dylan Wilson for helpful conversations.
    The second-named author gratefully acknowledges support from both the MIT Dean of Science Fellowship and the National Science Foundation Graduate Research Fellowship under Grant \#112237.
\end{acknowledgements}


\subsection{Notation \& background}

In this subsection we set the basic notational conventions that we use throughout this note as well as recall a bit of relevant background.

\begin{notation}
    Let $ \dX $ be an $ \infty $-category.
    If $ \dX $ has a terminal object, we write $ \ast \in \dX $ for the terminal object and $ \dX_{\ast} $ for the $ \infty $-category of pointed objects in $ \dX $.
    If $ \dX_{\ast} $ has coproducts and $ X,Y \in \dX_{\ast} $, we write $ X \wedgesum Y $ for the coproduct of $ X $ and $ Y $ in $ \dX_{\ast} $.
    If $ \dX_{\ast} $ has coproducts and products, note that there is a natural comparison morphism $ X \wedgesum Y \to X \times Y $ induced by the morphisms
    \begin{equation*}
        (\id_{X},\ast) \colon X \to X \times Y \andeq (\ast, \id_{Y}) \colon Y \to X \times Y \period
    \end{equation*}

    We say that a morphism $ f \colon X \to Y $ in $ \dX_{\ast} $ is \defn{null} if $ f $ factors through the zero object $ \ast $. 
\end{notation}

\begin{notation}
    Let $ \dX $ be an $ \infty $-category with pushouts. 
    Given a span $ X \leftarrow W \to Y $ in $ \dX $, we write $ X \coproduct^{W} Y $ for its pushout.
\end{notation}

\begin{recollection}
    Let $ \dX $ be an $ \infty $-category with finite products and pushouts, and $ X,Y \in \dX_{\ast} $ pointed objects of $ \dX $.
    The \defn{smash product} $ X \smashprod Y $ of $ X $ and $ Y $ is the cofiber
    \begin{equation*}
        \begin{tikzcd}[sep=2em]
            X \wedgesum Y \arrow[r] \arrow[d] \arrow[dr, phantom, very near end, "\ulcorner", xshift=0.25em, yshift=-0.25em] & X \times Y \arrow[d] \\
            \ast \arrow[r] & X \smashprod Y
        \end{tikzcd}
    \end{equation*}
    of the comparison morphism $ X \wedgesum Y \to X \times Y $. 
\end{recollection}

\begin{recollection}
    Let $ \dX $ be an $ \infty $-category with pushouts and a terminal object.
    The \defn{suspension} of an object $ X \in \dX $ is the pushout
    \begin{equation*}
        \begin{tikzcd}[sep=2em]
            X \arrow[r] \arrow[d] \arrow[dr, phantom, very near end, "\ulcorner", xshift=0.25em, yshift=-0.25em] & \ast \arrow[d] \\
            \ast \arrow[r] & \Sigma X \period
        \end{tikzcd}
    \end{equation*}
\end{recollection}

\begin{recollection}
    Let $ \dX $ be an $ \infty $-category with finite limits.
    The \defn{loop object} of a pointed object $ X \in \dX_{\ast} $ is the pullback
    \begin{equation*}
        \begin{tikzcd}[sep=2em]
            \Omega X \arrow[r] \arrow[d] \arrow[dr, phantom, very near start, "\lrcorner", xshift=-0.25em, yshift=0.25em] & \ast \arrow[d] \\
            \ast \arrow[r] & X
        \end{tikzcd}
    \end{equation*}
    in $ \dX_{\ast} $.
\end{recollection}

We also make repeated use of the following easy fact.
The unfamiliar reader should consult \cites[\S2]{Chromfracture:BarthelAntolin}{MO:333239}.

\begin{lemma}\label{lem:3by3}
    Let $ \dX $ be an $ \infty $-category with pushouts and 
    \begin{equation}\label{daig:3by3}
        \begin{tikzcd}[sep=2em]
            X_1 & X_0 \arrow[l] \arrow[r] & X_2 \\
            W_1 \arrow[u] \arrow[d] & W_0 \arrow[l] \arrow[r] \arrow[u] \arrow[d] & W_2 \arrow[u] \arrow[d] \\
            Y_1 & Y_0 \arrow[l] \arrow[r] & Y_2
        \end{tikzcd}
    \end{equation}
    a commutative diagram in $ \dX $.
    Then the colimit of the diagram \eqref{daig:3by3} exists and is equivalent to both of the following two iterated pushouts:
    \begin{enumerate}[{\upshape (\ref*{lem:3by3}.1)}]
        \item Form the pushout of the rows of \eqref{daig:3by3}, then take the pushout of the resulting span
        \begin{equation*}
            \begin{tikzcd}[sep=2em]
                \displaystyle X_1 \coproductlimits^{X_0} X_2 & \displaystyle W_1 \coproductlimits^{W_0} W_2 \arrow[l] \arrow[r] & \displaystyle Y_1 \coproductlimits^{Y_0} Y_2 \period
            \end{tikzcd}
        \end{equation*}

        \item Form the pushout of the columns of \eqref{daig:3by3}, then take the pushout of the resulting span
        \begin{equation*}
            \begin{tikzcd}[sep=2em]
                \displaystyle X_1 \coproductlimits^{W_1} Y_1 & \displaystyle X_0 \coproductlimits^{W_0} Y_0 \arrow[l] \arrow[r] & \displaystyle X_2 \coproductlimits^{W_2} Y_2 \period
            \end{tikzcd}
        \end{equation*}
    \end{enumerate}
\end{lemma}


\section{The James Splitting}\label{james-section}

In this section, we present a proof of the James Splitting which holds in any $\infty$-category with finite limits and pushouts, where pushout squares remain pushouts after basechange along an arbitrary morphism.
The argument we give roughly follows the argument Hopkins gave in his course on stable homotopy theory in the setting of spaces \cite[Lecture 4, \S 3]{hopkins-course-notes}; Hopkins attributes this proof to James \cites{MR73181}{MR77922}{MR79263} and Ganea \cite{MR229239}.


\subsection{Universal pushouts and Mather's Second Cube Lemma}\label{subsec:universalpushout}

The key property utilized in the proofs we present of the James and Hilton--Milnor Splittings is that pushout squares are preserved by arbitrary basechange. 
This implies that, in particular, the James and Hilton--Milnor Splittings
hold in any $ \infty $-topos, but also in other situations (such as motivic
spaces).
In this subsection, we provide the categorical context that we work in for the rest of the paper and give a convenient reformulation of the stability of pullbacks under basechange in terms of Mather's Second Cube Lemma (\Cref{Mather}).

\begin{recollection}
    Let $ \dI $ be an $ \infty $-category and let $ \dX $ be an $ \infty $-category with pullbacks and all $ \dI $-shaped colimits.
    We say that \defn{$ \dI $-shaped colimits in $ \dX $ are universal} if $ \dI $-shaped colimits in $ \dX $ are stable under pullback along any morphism. 
    That is, for every diagram $ F \colon \dI \to \dX $ and pair of morphisms $ \colim_{i \in \dI} F(i) \to Z $ and $ Y \to Z $ in $ \dX $, the natural morphism
    \begin{equation*}
        \colim_{i \in \dI} (F(i) \times_Z Y) \to \Big(\colim_{i \in \dI} F(i)\Big) \times_Z Y
    \end{equation*}
    is an equivalence.
\end{recollection}

\begin{example}
    Let $0\leq n \leq \infty$, and let $ \dX $ be an $ n $-topos.
    One of the Giraud--Lurie axioms for $ n $-topoi guarentees that all small colimits in $ \dX $ are universal \cite[Theorem 6.1.0.6 \& Proposition 6.4.1.5]{HTT}.
    In particular, small colimits in the category $ \Set $ of sets and the $ \infty $-category $ \Spc $ of spaces are universal.
\end{example}

\begin{example}
    Let $ \dX $ be an \textit{elementary $ \infty $-topos} in the sense of \cite[Definition 3.5]{Rasekh:elementary}.
    Then all colimits that exist in $ \dX $ are universal \cite[Corollary 3.12]{Rasekh:elementary}.
    In particular, finite colimits are universal in $ \dX $.
\end{example}

\begin{example}[motivic spaces]\label{ex:motivicspaces}
    Let $ S $ be a scheme.
    The $ \infty $-category $\HH(S)$ of \defn{motivic spaces over $ S $} is defined as the $ \AA^1 $-localization of the $\infty$-topos $ \Shnis(\Sm_S) $ of sheaves of spaces on the category $\Sm_S$ of smooth schemes of finite type over $S$ equipped with the Nisnevich topology.
    Concretely, $ \HH(S) $ is the full subcategory of $ \Shnis(\Sm_S) $ spanned by those Nisnevich sheaves $ \dF $ on $ \Sm_S $ with the property that for each smooth $ S $-scheme $ X $, the projection $ \pr_1 \colon X \times_S \AA_S^1 \to X $ induces an equivalence
    \begin{equation*}
       \pr_1^{\ast} \colon \dF(X) \isomto \dF(X \times_S \AA_S^1) \period
    \end{equation*}

    The inclusion $ \HH(S) \subset \Shnis(\Sm_S) $ admits a left adjoint $ \Lmot \colon \Shnis(\Sm_S) \to \HH(S) $ called \textit{motivic localization}.
    Motivic localization preserves finite products, but not all finite limits.
    Moreover, the $ \infty $-category $ \HH(S) $ is not an $ \infty $-topos (see \cites[Remark 3.5]{MR2916287}[\S4.3]{MR3890956}), and it is not immediately clear from the construction if any colimits are universal in $ \HH(S) $.
    Nonetheless, Hoyois has shown that all small colimits are universal in $ \HH(S) $ \cite[Proposition 3.15]{MR3570135}.
\end{example}

\begin{example}[profinite spaces]\label{ex:profinitespaces}
    We say that a space $ X $ is \defn{$ \uppi $-finite} if $ X $ is truncated, has finitely many connected components, and $ \uppi_i(X,x) $ is finite for every integer $ i \geq 1 $ and point $ x \in X $. 
    Write $ \Spc_{\uppi} \subset \Spc $ for the full subcategory spanned by the $ \uppi $-finite spaces and $ \Pro(\Spc_{\uppi}) $ for the $ \infty $-category of \defn{profinite spaces}.
    Infinite coproducts in $ \Pro(\Spc_{\uppi}) $ are not universal \SAG{Warning}{E.6.0.9}, however, finite colimits and geometric realizations of simplicial objects are universal in $ \Pro(\Spc_{\uppi}) $ \cite[\href{http://www.math.ias.edu/~lurie/papers/SAG-rootfile.pdf\#theorem.E.6.3.1}{Theorem E.6.3.1} \& \href{http://www.math.ias.edu/~lurie/papers/SAG-rootfile.pdf\#theorem.E.6.3.2}{Corollary E.6.3.2}]{SAG}.
\end{example}

The following result gives a reformulation of what it means for pushouts to be universal in terms of Mather's Second Cube Lemma, which Mather originally proved in the $ \infty $-category of spaces \cite[Theorem 25]{MR402694}.
See \kerodon{011H} for an elegant proof of Mather's Second Cube Lemma in $ \Spc $.

\begin{lemma}\label{Mather}
    Let $ \dX $ be an $ \infty $-category with pullbacks and pushouts.
    The following conditions are equivalent:
    \begin{enumerate}[{\upshape (\ref*{Mather}.1)}]
        \item\label{Mather.1} Pushouts in $ \dX $ are universal.

	    \item\label{Mather.2} \emph{Mather's Second Cube Lemma} holds in $\dX$:
	    Given a commutative cube 
        \begin{equation*}\label{eq:Mathercube}
            \begin{tikzcd}[column sep={6ex,between origins}, row sep={6ex,between origins}]
                A_0 \arrow[rr] \arrow[dd]  \arrow[dr] & & A_2 \arrow[dd]  \arrow[dr] \\
                & A_1 \arrow[rr, crossing over] & & A_3 \arrow[dd]  \\
                B_0 \arrow[rr] \arrow[dr] & & B_2 \arrow[dr] \\
                & B_1 \arrow[rr] \arrow[from=uu, crossing over] & & B_3  & 
            \end{tikzcd}
        \end{equation*}
        in $ \dX $ where the bottom horizontal face is a pushout square and all vertical faces are pullback squares, then the top horizontal square is a pushout square.
    \end{enumerate}
\end{lemma}

\begin{proof}
    The implication \enumref{Mather}{1} $ \Rightarrow $ \enumref{Mather}{2} is immediate.
    To see that \enumref{Mather}{2} $ \Rightarrow $ \enumref{Mather}{1}, suppose that we are given a pushout square
    \begin{equation}\label{sq:Bpullback}
        \begin{tikzcd}[sep=2em]
            B_0 \arrow[r] \arrow[d] \arrow[dr, phantom, very near end, "\ulcorner", xshift=0.25em, yshift=-0.25em] & B_2 \arrow[d] \\
            B_1 \arrow[r] & B_3
        \end{tikzcd}
    \end{equation} 
    in $ \dX $ and morphisms $ f \colon B_3 \to Z $ and $ g \colon Y \to Z $ in $ \dX $.
    For each $ i \in \{0,1,2,3\} $, define $ A_i \colonequals B_i \times_Z Y $, so that all the vertical squares in the diagram 
    \begin{equation}\label{diag:mathercubeuniversal}
        \begin{tikzcd}[column sep={6ex,between origins}, row sep={6ex,between origins}]
            A_0 \arrow[rr] \arrow[dd]  \arrow[dr] & & A_2 \arrow[dd]  \arrow[dr] \\
            & A_1 \arrow[rr, crossing over] & & A_3 \arrow[dd] \arrow[rr] & & Y \arrow[dd, "g"]  \\
            B_0 \arrow[rr] \arrow[dr] & & B_2 \arrow[dr] \\
            & B_1 \arrow[rr] \arrow[from=uu, crossing over] & & B_3 \arrow[rr, "f"'] & & Z 
        \end{tikzcd}
    \end{equation}
    are pullbacks.
    Since the bottom horizontal square of the cube in \eqref{diag:mathercubeuniversal} is a pushout, \enumref{Mather}{2} implies that the top horizontal square is also a pushout.
    Thus the pushout square \eqref{sq:Bpullback} remains a pushout after base change along an arbitrary morphism, as desired.
\end{proof}

Since the main results of this note are about pointed objects, we make the following mildly abusive convention:

\begin{convention}
    We say that an $ \infty $-category $ \dX $ \defn{has universal pushouts} if $ \dX $ has finite limits and pushouts, and pushouts in $ \dX $ are universal.
\end{convention}


\subsection{Statement of the James Splitting \& Consequences}\label{subsec:Jamessplitting}

The James Splitting, originally proven in \cite{MR73181}, provides a splitting of the space $ \Omega \Sigma X $ after a single suspension.
The goal of this subsection is to provide a proof of the James Splitting that only relies on the universality of pushouts and a few elementary computations involving the interaction between forming suspensions, loop objects, and smash products.

\begin{theorem}[Fundamental James Splitting]\label{james}
    Let $\dX$ be an $\infty$-category with universal pushouts. 
    For every pointed object $X\in \dX_{\ast}$, there is a natural equivalence
    \begin{equation*}
	   \Sigma \Omega \Sigma X \simeq \Sigma X \vee \Sigma(X\wedge \Omega \Sigma X) \period
    \end{equation*}
\end{theorem}

\noindent Using the fact that $ \Sigma(X \smashprod \Omega\Sigma X) \equivalent X \smashprod \Sigma\Omega\Sigma X $ (\Cref{lem:suspensionsmash}) and iterating the equivalence of \Cref{james}, we see:

\begin{corollary}[Fundamental James Splitting, redux]\label{jamesredux}
    Let $\dX$ be an $\infty$-category with universal pushouts. 
    For each pointed object $X\in \dX_{\ast}$ and integer $ n \geq 1 $, there is a natural equivalence
    \begin{equation*}
        \Sigma \Omega \Sigma X \simeq \paren{\bigvee_{1 \leq i \leq n} \Sigma X^{\wedge i}} \vee (X^{\smashprod n} \wedge \Sigma\Omega \Sigma X) \period
    \end{equation*}
\end{corollary}

\begin{notation}\label{ntn:Jamescomparison}
    Let $\dX$ be an $\infty$-category with universal pushouts and $X\in \dX_{\ast}$ a pointed object.
    Assume that $ \dX_{\ast} $ has countable coproducts.
    Passing to the colimit as $ n \to \infty $, the coproduct insertions
    \begin{equation*}
        \bigvee_{i = 1}^{n} \Sigma X^{\wedge i} \to \Sigma \Omega \Sigma X
    \end{equation*}
    provided by \Cref{jamesredux} provide a natural comparison morphism
    \begin{equation*}
        \bigvee_{i = 1}^{\infty} \Sigma X^{\wedge i} \to \Sigma \Omega \Sigma X \comma
    \end{equation*}
    which we denote by $ c_X \colon \bigvee_{i = 1}^{\infty} \Sigma X^{\wedge i} \to \Sigma \Omega \Sigma X  $.
\end{notation}

\begin{warning}
    The comparison morphism $ c_X $ need not be an equivalence.
    For example, if $ \dX = \Spc $ and $ X = \Sph{0} $ is the $ 0 $-sphere, then the map
    \begin{equation*}
        c_{\Sph{0}} \colon \bigvee_{i = 1}^{\infty} \Sph{1} \to \Sigma \Omega \Sigma \Sph{0} \equivalent \bigvee_{i \in \ZZ} \Sph{1} 
    \end{equation*}
    is not an equivalence.
    Even though both the source and target of \smash{$ c_{\Sph{0}} $} are both countable wedges of copies of $ \Sph{1} $, the map $ c_{\Sph{0}} $ is the summand inclusion induced by the inclusion $ \ZZ_{\geq 1} \subset \ZZ $.

    We analyze when the comparison morphism $ c_X $ is an equivalence in \cref{subsec:infsplitting}.
\end{warning}

\begin{example}\label{motivic-james}
    Let $ S $ be a scheme.
    Since colimits are universal in the $ \infty $-category $\HH(S)$ of motivic spaces over $ S $ (\Cref{ex:motivicspaces}), \Cref{james} implies that for any pointed motivic space $ X \in \HH(S)_{\ast} $ and integer $ n \geq 1 $, we have $ \Sph{1} $-James Splittings
    \begin{equation*}
        \Sigma \Omega \Sigma X \simeq \paren{\bigvee_{1 \leq i \leq n} \Sigma X^{\wedge i}} \vee (X^{\smashprod n} \wedge \Sigma\Omega \Sigma X) \period
    \end{equation*}

    Write $ \GGm $ for the multiplicative group scheme over $ S $.
    Since $ \Sigma\GGm \equivalent \PP_S^1 $, setting $ X = \GGm $ we see that 
    \begin{equation}\label{eq:P1James}
        \Sigma \Omega \PP_S^1 \simeq \paren{\bigvee_{1 \leq i \leq n} \Sigma \GGm^{\wedge i}} \vee (\GGm^{\smashprod n} \wedge \Sigma\Omega\PP_S^1) \period
    \end{equation}
    Using the grading convention $ \Sph{a,b} \colonequals \GGm^{\smashprod b} \smashprod (\Sph{1})^{\smashprod(a-b)} $ for motivic spheres, we can rewrite the equivalence \eqref{eq:P1James} as
    \begin{equation*}
        \Sigma \Omega \PP_S^1 \simeq \paren{\bigvee_{1 \leq i \leq n} \Sph{i+1,i}} \vee (\Sph{n+1,n} \wedge \Omega\PP_S^1) \period
    \end{equation*}
\end{example}

\begin{remark}
    There is another suspension in motivic homotopy theory, given by smashing with the multiplicative group scheme \smash{$ \GGm $}.
    One would like an analogue of the James Splitting in $\HH(S)_{\ast}$ for \smash{$ \GGm $}-suspensions.
    For $ S = \spec(\RR)$, Betti realization defines a functor \smash{$ \HH(\spec(\RR))\to \Spc_{\upC_2} $} to $\upC_2$-spaces which sends \smash{$\GGm$} to the sign representation circle \smash{$ \Sph{\upsigma}$} and $ \Sph{1} $ to the circle with trivial $ \upC_2 $-action.
    Even though Betti realization is not an equivalence, it closely ties $\RR$-motivic
    homotopy theory with $\upC_2$-equivariant homotopy theory.
    In \cite{hill-james}, Hill studies the \textit{signed James construction} in $\upC_2$-equivariant unstable homotopy theory, and shows that an analogue of the James Splitting holds for $\Omega^{\upsigma} \Sigma^{\upsigma} X$ after suspending by the regular representation sphere \smash{$ \Sph{\uprho} = \Sph{1} \wedge \Sph{\upsigma} $}.
    This might lead one to hope that there is an analogue of Hill's result in motivic homotopy theory which proves the James Splitting for \smash{$\Omega_{\GGm} \Sigma_{\GGm} X$} after $\PP^1$-suspension; at the moment, we are not aware of such a result.
\end{remark}


\subsection{Proof of the James Splitting}\label{subsec:Jamessplittingproof}

Before we prove \Cref{james}, we need a few preliminary results.
First, we give a convenient expression for $ \Sigma\Omega\Sigma X $ as the cofiber of the projection $ \pr_2 \colon X \times \Omega\Sigma X \to \Omega\Sigma X $.
This expression for $ \Sigma\Omega\Sigma X $ is an immediate consequence of the following:

\begin{lemma}\label{lem:pr2pushout}
    Let $\dX$ be an $\infty$-category with universal pushouts.
    For every pointed object $ X \in \dX_{\ast} $, there exists a natural morphism $ a_X \colon X \times \Omega\Sigma X \to \Omega\Sigma X $ and a pushout square 
    \begin{equation*}
        \begin{tikzcd}[sep=2em]
            X \times \Omega\Sigma X \arrow[r, "\pr_2"] \arrow[d, "a_X"'] \arrow[dr, phantom, very near end, "\ulcorner", xshift=0.45em, yshift=-0.15em] & \Omega\Sigma X \arrow[d] \\
            \Omega\Sigma X \arrow[r] & \ast \period
        \end{tikzcd}
    \end{equation*}
\end{lemma}

\begin{proof}
    Write
    \begin{equation*}
        \begin{tikzcd}[sep=2em]
            X \arrow[r] \arrow[d] \arrow[dr, phantom, very near end, "\ulcorner", xshift=0.25em, yshift=-0.25em] & \ast \arrow[d, "i_2"] \\
            \ast \arrow[r, "i_1"'] & \Sigma X 
        \end{tikzcd}
    \end{equation*}
    for the pushout square defining the suspension $ \Sigma X $.
    The definition of $ \Sigma X $ provides an equivalence between the points $ i_1,i_2 \colon \ast \to \Sigma X $, hence there are natural pullback squares
    \begin{equation*}
        \begin{tikzcd}[sep=2em]
            X \times \Omega\Sigma X \arrow[r, "a_X"] \arrow[d] \arrow[dr, phantom, very near start, "\lrcorner", xshift=-0.5em, yshift=0.2em] & \Omega X \arrow[r] \arrow[d] \arrow[dr, phantom, very near start, "\lrcorner", xshift=-0.2em, yshift=0.2em] & \ast \arrow[d, "i_2"] \\
            X \arrow[r] & \ast \arrow[r, "i_1"'] & X
        \end{tikzcd}
    \end{equation*}
    The claim now follows from Mather's Second Cube Lemma applied to the cube
    \begin{equation*}\label{eq:Mathercubeapp}
        \begin{gathered}[b]
            \begin{tikzcd}[column sep={8ex,between origins}, row sep={8ex,between origins}]
                X \times \Omega\Sigma X \arrow[rr, "\pr_2"] \arrow[dd, near end]  \arrow[dr, "a_X" description] & & \Omega\Sigma X \arrow[dd]  \arrow[dr] \\
                & \Omega\Sigma X \arrow[rr, crossing over] & & \ast \arrow[dd, "i_2"]  \\
                X \arrow[rr] \arrow[dr] & & \ast \arrow[dr, "i_2" description] \\
                & \ast \arrow[rr, "i_1"'] \arrow[from=uu, crossing over] & & \Sigma X  \period 
            \end{tikzcd}
            \\
        \end{gathered}
        \qedhere
    \end{equation*}
\end{proof}

\begin{warning}
    The morphism $ a_X \colon X \times \Omega\Sigma X \to \Omega\Sigma X $ in \Cref{lem:pr2pushout} cannot generally be identified with the second projection $ \pr_2 \colon X \times \Omega\Sigma X \to \Omega\Sigma X $.
    Indeed, since $ \dX $ is assumed to have universal pushouts, there is a natural pushout square
    \begin{equation*}
        \begin{tikzcd}[sep=2em]
            X \times \Omega\Sigma X \arrow[r, "\pr_2"] \arrow[d, "\pr_2"'] \arrow[dr, phantom, very near end, "\ulcorner", xshift=1em, yshift=-0.25em] & \Omega\Sigma X \arrow[d] \\
            \Omega\Sigma X \arrow[r] & \Sigma X \times \Omega\Sigma X \period
        \end{tikzcd}
    \end{equation*}
    Moreover, the object $ \Sigma X \times \Omega\Sigma X $ is not generally terminal in $ \dX $.
\end{warning}

\begin{corollary}\label{cor:cofibpr2}
    Let $\dX$ be an $\infty$-category with universal pushouts.
    For every pointed object $ X \in \dX_{\ast} $, there is a natural equivalence 
    \begin{equation*}
        \cofib(\pr_2 \colon X \times \Omega\Sigma X \to \Omega\Sigma X) \equivalent \Sigma\Omega\Sigma X \period
    \end{equation*} 
\end{corollary}

Next, we give a convenient expression for the term $ \Sigma(X \wedge \Omega \Sigma X) $ in the James Splitting as the pushout of the span
\begin{equation*}
    \begin{tikzcd}[sep=2em]
        X & X \times \Omega \Sigma X \arrow[l, "\pr_1"'] \arrow[r, "\pr_2"] & \Omega \Sigma X \period 
    \end{tikzcd}
\end{equation*}
Our proof of this appeals to the following fact, which follows immediately from the definitions.

\begin{lemma}\label{lem:wedgepushout}
    Let $ \dX $ be an $ \infty $-category with pushouts and a terminal object, and let $ X,Y \in \dX_{\ast} $ be pointed objects of $ \dX $.
    Then the square 
    \begin{equation*}
        \begin{tikzcd}[sep=2em]
            X \wedgesum Y \arrow[r, "{(*,\id_{Y})}"] \arrow[d, "{(\id_{X},*)}"'] & Y \arrow[d] \\
            X \arrow[r] & \ast
        \end{tikzcd}
    \end{equation*}
    is a pushout square.
\end{lemma}

\begin{prop}\label{pushout}
    Let $\dX$ be an $\infty$-category with finite limits and pushouts.
    Then for every pair of pointed objects $ X,Y \in \dX $, there is a pushout
    square
    \begin{equation*}
        \begin{tikzcd}[sep=2em]
            X \times Y \arrow[r, "\pr_2"] \arrow[d, "\pr_1"'] & Y \arrow[d] \\
            X \arrow[r] & \Sigma(X \smashprod Y) \comma
        \end{tikzcd}
    \end{equation*}
    where the morphisms $ X \to \Sigma(X \smashprod Y) $ and $ Y \to \Sigma(X
    \smashprod Y) $ are null.
\end{prop}

\begin{proof}
    Let $C$ denote the pushout $ X \coproductlimits^{X \times Y} Y $; we desire to show that $ C \simeq \Sigma (X\wedge Y) $.
    We apply \Cref{lem:3by3} to the commutative diagram
    \begin{equation}\label{iterate}
        \begin{tikzcd}[sep=2em]
    	    \ast & \ast \arrow[l] \arrow[r] & \ast \\
            X \arrow[u] \arrow[d, equals] & X\vee Y \arrow[l] \arrow[r] \arrow[u] \arrow[d] & Y \arrow[u] \arrow[d, equals] \\
    	    X & X\times Y \arrow[l] \arrow[r] & Y \period
        \end{tikzcd}
    \end{equation}
    Appealing to \Cref{lem:wedgepushout}, taking pushouts of the rows of \eqref{iterate} results in the span
    \begin{equation*}
        \begin{tikzcd}[sep=1.5em]
            C & \ast \arrow[l] \arrow[r] & \ast \comma 
        \end{tikzcd}
    \end{equation*}
    which has pushout $ C $.
    Alternatively, since the smash product $ X \wedge Y$ is the cofiber of the comparison morphism $ X \vee Y \to X\times Y $, taking pushouts of the columns of \eqref{iterate} results in the span
    \begin{equation}\label{sigmasmashspan}
        \begin{tikzcd}[sep=1.5em]
            \ast & X\smashprod Y  \arrow[l] \arrow[r] & \ast \period 
        \end{tikzcd}
    \end{equation}
    By definition, the pushout of the span \eqref{sigmasmashspan} is the suspension $ \Sigma(X\wedge Y) $, so \Cref{lem:3by3} shows that
    \begin{equation*}
        C \equivalent \Sigma (X\wedge Y) \period
    \end{equation*}
   
    To conclude the proof, note that it follows from the definitions that the induced morphisms
    \begin{equation*}
        X \to \Sigma(X \wedge Y) \andeq Y \to \Sigma(X \wedge Y)
    \end{equation*}
    factor through the zero object $ \ast \in \dX_{\ast} $.
\end{proof}

\Cref{pushout} also provides a general formula for the cofiber $ \cofib(\pr_2 \colon X \times Y \to Y) $ that allows us to relate the expressions for $ \Sigma\Omega\Sigma X $ and $ \Sigma(X \wedge \Omega \Sigma X) $ from \Cref{cor:cofibpr2,pushout}, respectively.

\begin{corollary}\label{cofiber-product}
    Let $\dX$ be an $\infty$-category with finite limits and pushouts.
    Then, for every pair of pointed objects $ X,Y\in \dX_{\ast} $:
    \begin{enumerate}[{\upshape (\ref*{cofiber-product}.1)}]
        \item\label{cofiber-product.1} There is a natural equivalence $ \cofib(\pr_2 \colon X \times Y \to Y) \equivalent \Sigma X \wedgesum \Sigma(X \smashprod Y) $.

        \item\label{cofiber-product.2} There a natural equivalence $ \Sigma(X\times Y) \simeq \Sigma(X \wedge Y) \vee \Sigma X \vee \Sigma Y $.
    \end{enumerate}
\end{corollary}

\begin{proof}
    Consider the diagram
    \begin{equation}\label{eq:4by4}
        \begin{tikzcd}[column sep={22ex,between origins}, row sep={8ex,between origins}]
            X \times Y \arrow[r, "\pr_2"] \arrow[d, "\pr_1"'] \arrow[dr, phantom, very near end, "\ulcorner", xshift=0.75em, yshift=-0.25em] & Y \arrow[d] \arrow[r] & \ast \arrow[d] \\
            X \arrow[r] \arrow[d] & \Sigma(X \smashprod Y) \arrow[d] \arrow[r] &  \Sigma Y \wedgesum \Sigma(X \smashprod Y) \arrow[d]  \\ 
            \ast \arrow[r] & \Sigma X \wedgesum \Sigma(X \smashprod Y) \arrow[r] &
	    \Sigma X \wedgesum \Sigma Y \wedgesum \Sigma(X \smashprod Y) \comma
        \end{tikzcd}
    \end{equation}
    where the top-left square is the pushout square of \Cref{pushout} and all of the morphisms in the bottom-right square are coproduct insertions.
    Since the maps $ X \to \Sigma(X \smashprod Y) $ and $ Y \to \Sigma(X \smashprod Y) $ are null, the diagram \eqref{eq:4by4} commutes and the bottom-left and top-right squares of \eqref{eq:4by4} are pushout squares.
    This proves \enumref{cofiber-product}{1}.
    To prove \enumref{cofiber-product}{2}, note that the bottom-right square in the diagram \eqref{eq:4by4} is a pushout.
\end{proof}

\Cref{cor:cofibpr2,cofiber-product} now combine to give the James Splitting.

\begin{proof}[Proof of \Cref{james}]
    Combining \Cref{cor:cofibpr2} with \Cref{cofiber-product} in the case that $ Y = \Omega\Sigma X$ we see that there are natural equivalences
    \begin{align*}
        \Sigma\Omega\Sigma X &\equivalent \cofib(\pr_2 \colon X \times \Omega\Sigma X \to \Omega\Sigma X) \\
        &\equivalent \Sigma X \wedgesum \Sigma( X \smashprod \Omega\Sigma X) \period \qedhere
    \end{align*} 
\end{proof}

The splitting
\begin{equation*}
    \Sigma \Omega \Sigma X \simeq \paren{\bigvee_{1 \leq i \leq n} \Sigma X^{\wedge i}} \vee (X^{\smashprod n} \wedge \Sigma\Omega \Sigma X)
\end{equation*}
of \Cref{jamesredux} is immediate from \Cref{james} combined with the following elementary fact:

\begin{lemma}\label{lem:suspensionsmash}
    Let $\dX$ be an $\infty$-category with universal pushouts.
    For every pair of pointed objects $ X,Y\in \dX_{\ast} $, there is a natural equivalence
    \begin{equation*}
        \Sigma(X \smashprod Y) \equivalent X \smashprod \Sigma Y \period
    \end{equation*}
\end{lemma}

\begin{proof}
    Since pushouts in $ \dX $ are universal and colimits commute, the squares
    \begin{equation*}
        \begin{tikzcd}[sep=2.5em]
           X \times Y \arrow[r, "\id_{X} \times \ast"] \arrow[d, "\id_{X} \times \ast"'] & X \times \ast \arrow[d, "\id_{X} \times \ast"] \\
           X \times \ast \arrow[r, "\id_{X} \times \ast"'] & X \times \Sigma Y
        \end{tikzcd}
        \andeq
        \begin{tikzcd}[sep=2.5em]
           X \wedgesum Y \arrow[r, "\id_{X} \wedgesum \ast"] \arrow[d, "\id_{X} \wedgesum \ast"'] & X \wedgesum \ast \arrow[d, "\id_{X} \wedgesum \ast"] \\
           X \wedgesum \ast \arrow[r, "\id_{X} \wedgesum \ast"'] & X \wedgesum \Sigma Y
        \end{tikzcd}
    \end{equation*}
    are both pushouts in $ \dX_{\ast} $.
    By the definition of the smash product and the facts that colimits commute and $ X \smashprod \ast \equivalent \ast $, we see that
    \begin{align*}
        X \smashprod \Sigma Y &= \cofib(X \vee \Sigma Y \to X \times \Sigma Y) \\
        &\equivalent \cofib\paren{ (X \wedgesum \ast) \coproductlimits^{X \wedgesum Y} (X \wedgesum \ast) \to (X \times \ast) \coproductlimits^{X \times Y} (X \times \ast) } \\ 
        &\equivalent (X \smashprod \ast) \coproductlimits^{X \smashprod Y} (X \smashprod \ast) \\
        &\equivalent \Sigma(X \smashprod Y) \period \qedhere
    \end{align*}
\end{proof}


\subsection{Ganea's Lemma}\label{subsec:Ganea}

Since the method of proof is similar to the arguments in this section, we close with the following lemma of Ganea \cite[Theorem 1.1]{MR179791}.
This will not be used in the sequel.

\begin{lemma}\label{lem:Ganea}
    Let $\dX$ be an $\infty$-category with universal pushouts.
    Let $ f \colon X \to Y $ be a morphism in $ \dX_{\ast} $, and write $ i \colon \fib(f) \to X $ for the induced morphism from the fiber of $ f $.
    Then there is a natural equivalence
    \begin{equation*}
        \fib(\cofib(i) \to Y) \equivalent \Sigma(\Omega Y \smashprod \fib(f)) \period
    \end{equation*}
\end{lemma}

\begin{proof}
    By \Cref{pushout}, it suffices to show that the square
    \begin{equation*}
        \begin{tikzcd}[sep=2em]
           \Omega Y \times \fib(f) \arrow[r, "\pr_2"] \arrow[d, "\pr_1"'] & \fib(f) \arrow[d] \\
           \Omega Y \arrow[r] & \fib(\cofib(i)\to Y)
        \end{tikzcd}
    \end{equation*}
    is a pushout.
    Consider the diagram
    \begin{equation}\label{diag:bigpullback}
        \begin{tikzcd}[column sep={10ex,between origins}, row sep={8ex,between origins}]
            \Omega Y \times \fib(f)  \arrow[rr, "\pr_2"] \arrow[dd, "\pr_2"']  \arrow[dr, "\pr_1" description] & & \fib(f) \arrow[dd] \arrow[dr] \\
            & \Omega Y \arrow[rr, crossing over] & & \fib(\cofib(i) \to Y) \arrow[dd] \arrow[rr] & & \ast \arrow[dd] \\
            \fib(f) \arrow[rr, "i" near end] \arrow[dr] & & X \arrow[dr] \\
            & \ast \arrow[rr] \arrow[from=uu, crossing over] & & \cofib(i) \arrow[rr] & & Y \comma
        \end{tikzcd}
    \end{equation}
    and note that each vertical square is a pullback square.
    The bottom horizontal square in \eqref{diag:bigpullback} is a pushout square by definition, so the assumption that pushouts in $ \dX $ are universal implies that the top horizontal square is a pushout as well.
\end{proof}


\section{The Hilton--Milnor Splitting}\label{sec:Hilton-Milnor}

The main result of this section is the following:

\begin{theorem}[Fundamental Hilton--Milnor Splitting]\label{hilton-milnor}
    Let $\dX$ be an $\infty$-category with universal pushouts and $X,Y\in \dX_{\ast}$.
    Then there is a natural equivalence
    \begin{equation*}
        \Omega(X \vee Y) \simeq \Omega X\times \Omega Y\times \Omega \Sigma(\Omega X \wedge \Omega Y) \period
    \end{equation*}
\end{theorem}

Before giving the proof of \cref{hilton-milnor}, we discuss some applications.

\begin{example}\label{motivic-hilton-milnor}
    Let $ S $ be a scheme.
    Since colimits are universal in the $ \infty $-category $\HH(S)$ of motivic spaces over $ S $ (\Cref{ex:motivicspaces}), \Cref{hilton-milnor} implies that for any pointed motivic spaces $ X,Y \in \HH(S)_{\ast} $ we have an equivalence
    \begin{equation*}
        \Omega(X \vee Y) \simeq \Omega X\times \Omega Y\times \Omega \Sigma(\Omega X \wedge \Omega Y) \period
    \end{equation*}
\end{example}

The Hilton--Milnor Splitting allows us to give a new description of the motivic space $ \Omega\Sigma(\PP^1 \smallsetminus \{0,1,\infty\}) $:

\begin{example}\label{example:P1threepts}
    Let $ S $ be a quasicompact quasiseparated scheme.
    Write \smash{$ \{0, 1, \infty\} \subset \PP_S^1 $} for the closed subscheme given by the union of the images of the closed embeddings \smash{$ S \hookrightarrow \PP_S^1 $} at $ 0 $, $ 1 $, and $ \infty $, respectively.
    Using the Morel--Voevodsky motivic purity Theorem \cites[Theorem 7.6]{MR3727503}[\S7.5]{MR3570135}[\S3, Theorem 2.23]{MR1813224}, Wickelgren \cite[Corollary 3.2]{MR3521594} showed that there is an equivalence
    \begin{equation*}
        \Sigma(\PP_S^1 \smallsetminus \{0, 1, \infty\}) \simeq \Sigma(\GGm \vee \GGm)
    \end{equation*}
    in the $ \infty $-category $ \HH(S)_{\ast} $.
    Since $ \Sigma \GGm \equivalent \PP_S^1 $, setting $ X = Y = \Sigma\GGm $ in \Cref{motivic-hilton-milnor} we see that there are equivalences
    \begin{align*}
        \Omega \Sigma(\PP_S^1 \smallsetminus \{0,1,\infty\}) &\simeq \Omega \Sigma(\GGm \vee
        \GGm) \\
        &\simeq \Omega \Sigma \GGm \times \Omega \Sigma \GGm \times
        \Omega \Sigma(\Omega\Sigma \GGm \smashprod \Omega\Sigma\GGm) \\
        &\simeq \Omega \PP_S^1 \times \Omega \PP_S^1 \times \Omega \Sigma(\Omega\PP_S^1 \smashprod \Omega\PP_S^1) \period
    \end{align*}
    Applying the James Splitting of \Cref{motivic-james} we see that for each integer $ n \geq 1 $, we can also express $ \Omega \Sigma(\PP_S^1 \smallsetminus \{0,1,\infty\}) $ as
    \begin{equation*}
        \Omega \Sigma(\PP_S^1 \smallsetminus \{0,1,\infty\}) \simeq \Omega \PP_S^1 \times \Omega \PP_S^1 \times \Omega \paren{ \paren{ \bigvee_{1 \leq i \leq n} \Omega\PP_S^1 \smashprod \Sph{i+1,i} } \wedgesum \paren{\Sph{n+1,n} \smashprod (\Omega\PP_S^1)^{\smashprod 2} } } \period
    \end{equation*}
\end{example}

We now turn to the proof of the Hilton--Milnor Splitting.
We first show that there is a fiber sequence 
\begin{equation}\label{seq:fibofcan}
    \begin{tikzcd}[sep=1.5em]
         \Sigma (\Omega Y \wedge \Omega X) \arrow[r] & X \wedgesum Y \arrow[r] & X \times Y \period
    \end{tikzcd}
\end{equation}
We then show that the sequence \eqref{seq:fibofcan} splits after taking loops.
To do this, we construct a section
\begin{equation*}
    \Omega(X \times Y) \to \Omega(X \wedgesum Y) \comma
\end{equation*}
and use the fact that a fiber sequence of group objects with a section splits on the level of underlying objects.
After proving that \eqref{seq:fibofcan} is a fiber sequence we give a quick review of group objects and deduce \Cref{hilton-milnor} from the Splitting Lemma (\Cref{lem:split}).

We start with the following observation:

\begin{lemma}\label{lem:fibofincintoprod}
    Let $ \dX $ be an $\infty$-category with finite limits and $ X, Y \in \dX_{\ast} $.
    Then there is a natural equivalence
    \begin{equation*}
        \fib((\id_{X},\ast) \colon X \to X \times Y ) \equivalent \Omega Y \period
    \end{equation*}
\end{lemma}

\noindent Next, we prove the existence of the fiber sequence \eqref{seq:fibofcan}.

\begin{lemma}\label{lem:fibwedgetoprod}
    Let $ \dX $ be an $\infty$-category with universal pushouts and $ X,Y \in \dX_{\ast} $.
    Then there is a natural equivalence
    \begin{equation*}
        \fib(X \vee Y \to X \times Y) \equivalent \Sigma (\Omega Y \wedge \Omega X) \period
    \end{equation*}
\end{lemma}

\begin{proof}   
    Write $ F \colonequals \fib(X \vee Y \to X \times Y) $.
    By \Cref{pushout}, it suffices to show that there is a pushout square
    \begin{equation*}
        \begin{tikzcd}[sep=2em]
           \Omega Y \times \Omega X \arrow[r, "\pr_2"] \arrow[d, "\pr_1"'] & \Omega X \arrow[d] \\
           \Omega Y \arrow[r] & F \period
        \end{tikzcd}
    \end{equation*}
    Consider the diagram
    \begin{equation}\label{diag:bigpullbackwedge}
        \begin{tikzcd}[column sep={8ex,between origins}, row sep={8ex,between origins}]
            \Omega Y \times \Omega X  \arrow[rr, "\pr_2"] \arrow[dd]  \arrow[dr, "\pr_1" description] & & \Omega X \arrow[dd] \arrow[dr] \\
            & \Omega Y \arrow[rr, crossing over] & & F \arrow[dd] \arrow[rr] & & \ast \arrow[dd] \\
            \ast \arrow[rr] \arrow[dr] & & Y \arrow[dr] \\
            & X \arrow[rr] \arrow[from=uu, crossing over] & & X \wedgesum Y \arrow[rr] & & X \times Y \period
        \end{tikzcd}
    \end{equation}
    The right-most vertical square in \eqref{diag:bigpullbackwedge} is a pullback by definition, and the front and right vertical squares in the cube appearing in \eqref{diag:bigpullbackwedge} are pullback squares by \Cref{lem:fibofincintoprod}.
    The back and left vertical squares in the cube appearing in \eqref{diag:bigpullbackwedge} are pullback squares by the Gluing Lemma for pullback squares.
    The bottom horizontal square in \eqref{diag:bigpullbackwedge} is a pushout square by definition, so the assumption that pushouts in $ \dX $ are universal implies that the top horizontal square is a pushout as well.
\end{proof}


\subsection{Reminder on group objects \& the Splitting Lemma}\label{subsec:groups}

In order to split the fiber sequence \eqref{seq:fibofcan} after taking loops, we need a few basic facts about \textit{group objects} (also called \textit{$ \Eone $-groups} or \textit{grouplike $ \Eone $-algebras}) in $ \infty $-categories, which we now review.
We begin with a little motivation for the definition of group objects as deloopings.
For the genesis of these ideas, we refer the reader to \cites{MR0339132}{MR232393}{MR353298}.
The reader should consult \cite[\S\S\HAsubseclink{4.1.2} \& \HAsubseclink{5.2.6}]{HA} for a modern treatment.

\begin{notation}
    We write $ \Deltab $ for the category of nonempty linearly ordered finite sets.
    As usual, given a simplicial object $ X \colon \Deltaop \to \dX $, we write $ X_n \colonequals X([n]) $ for the $ n $-simplices of $ X $.
\end{notation}

\noindent Recall that the bar construction is a fully faithful functor from the category of monoids to the category of simplicial sets.
The essential image of the bar construction consists of those simplicial sets $ X \colon \Deltaop \to \Set $ satisfying the following conditions:
\begin{enumerate}[(1)]
    \item\label{item:monoid.1} We have $ X_0 \equivalent \ast $.

    \item\label{item:monoid.2} \textit{Segal condition:} For each $ n > 0 $ and $ t \in [n] $, the square 
    \begin{equation*}
        \begin{tikzcd}[sep=2em]
            X([n]) \arrow[r] \arrow[d] & X(\{t < \cdots < n\}) \arrow[d] \\
            X(\{0 < \cdots < t\}) \arrow[r] & X(\{t\})
        \end{tikzcd}
    \end{equation*}
    is a pullback square.
\end{enumerate}
The face map $ d_1 \colon X_1 \times X_1 \equivalent X_2 \to X_1 $ provides a multiplication on $ X_1 $ with unit given by the degeneracy map $ s_0 \colon \ast \equivalent X_0 \to X_1 $.

Since groups form a full subcategory of the category of monoids, the bar construction also identifies the category of groups with a full subcategory of the category of simplicial sets.
For this it is better to use an alternative characterization of the existence of inverses: a monoid $ M $ is a group if and only if the \textit{shear maps}
\begin{align*}
    M \times M &\to M \times M \andeq M \times M \to M \times M\\ 
    (m,n) &\mapsto (m,mn) \phantom{\andeq} (m,n) \mapsto (mn,n)
\end{align*}
are bijections.
Translating this into simplicial sets one sees that the category of groups is equivalent to the full subcategory of $ \Fun(\Deltaop,\Set) $ spanned by the simplicial sets $ X $ satisfying \eqref{item:monoid.1}, \eqref{item:monoid.2}, and:
\begin{enumerate}[(1)]
   \setcounter{enumi}{2}

    \item\label{item:monoid.3} The induced squares
    \begin{equation*}
        \begin{tikzcd}[sep=2em]
            X(\{0 < 1 < 2\}) \arrow[r, "d_1"] \arrow[d, "d_2"'] & X(\{0 < 2\}) \arrow[d] \\
            X(\{0 < 1\}) \arrow[r] & X(\{0\})
        \end{tikzcd} 
        \andeq
         \begin{tikzcd}[sep=2em]
            X(\{0 < 1 < 2\}) \arrow[r, "d_1"] \arrow[d, "d_0"'] & X(\{0 < 2\}) \arrow[d] \\
            X(\{1 < 2\}) \arrow[r] & X(\{0\})
        \end{tikzcd} 
    \end{equation*}
    are pullback squares.
\end{enumerate}
We emphasize that condition \eqref{item:monoid.3} is not implied by the Segal condition \eqref{item:monoid.2}.

The following is the correct generalization of a group object in an arbitrary $ \infty $-category.
The point is to replace the Segal condition with a stronger condition that also encompasses condition \eqref{item:monoid.3}.
See \cites[Definitions \HTTthmlink{6.1.2.7} \& \HTTthmlink{7.2.2.1}]{HTT}.

\begin{definition}\label{def:group}
    Let $ \dX $ be an $ \infty $-category.
    A \defn{group object} in $ \dX $ is a simplicial object $ G \colon \Deltaop \to \dX $ such that:
    \begin{enumerate}[(\ref*{def:group}.1)]
        \item The object $ G_{0} $ is a terminal object of $ \dX $.

        \item For each object $ S \in \Deltaop $ and partition $ S = T \cup T' $ such that $ T \cap T' = \{t\} $ consists of a single element, the induced square
        \begin{equation*}
            \begin{tikzcd}[sep=2em]
                G(S) \arrow[r] \arrow[d] & G(T') \arrow[d] \\
                G(T) \arrow[r] & G(\{t\})
            \end{tikzcd}
        \end{equation*}
        is a pullback square in $ \dX $.
    \end{enumerate}
    In this case, we call $ G_1 \in \dX $ the \defn{underlying object} of $ G $.
    We often identify a group object by its underlying object.
    We write $ \Grp(\dX) \subset \Fun(\Deltaop,\dX) $ for the full subcategory spanned by the group objects.
\end{definition}

The key example of a group object is loops on a pointed object.
As a simplicial object, $ \Omega X $ can be written as the \textit{Čech nerve} of the basepoint $ \ast \to X $; since we use Čech nerves in \cref{subsec:connectedness}, we recall the definition here.

\begin{recollection}\label{cech-defn}
    Let $ \dX $ be an $ \infty $-category with pullbacks, and let $ e \colon W \to X $ be a morphism in $ \dX $.
    The \defn{Čech nerve} $ \Cech(e) $ of $ e $ is the simplicial object
    \begin{equation*}
        \begin{tikzcd}[sep=1.5em]
            \cdots \arrow[r, shift left=0.75ex] \arrow[r, shift right=0.75ex] \arrow[r, shift right=2.25ex] \arrow[r, shift left=2.25ex] & \displaystyle W \timeslimits_X W \timeslimits_X W \arrow[l] \arrow[l, shift left=1.5ex] \arrow[l, shift right=1.5ex] \arrow[r] \arrow[r, shift left=1.5ex] \arrow[r, shift right=1.5ex] & \displaystyle W \timeslimits_X W \arrow[l, shift left=0.75ex] \arrow[l, shift right=0.75ex] \arrow[r, shift left=0.75ex] \arrow[r, shift right=0.75ex] & W \arrow[l]
          \end{tikzcd}
    \end{equation*} 
    in $ \dX $.
    Here $ \Cech(e)_n $ is the $ (n+1) $-fold fiber product of $ W $ over $ X $, each degeneracy map is a diagonal morphism, and each face map is a projection.
    Note that the morphism $ e \colon W \to X $ defines a natural augmentation $ \Cech(e) \to X $.
\end{recollection}

\begin{lemma}\label{loop-monoid}
    Let $\dX$ be an $\infty$-category with finite limits and $ X\in \dX_{\ast} $.
    Then $\Omega X$ naturally admits the structure of a group object of $\dX$.
\end{lemma}

\begin{proof}
    Let $ \upU(X) $ denote the Čech nerve of the basepoint $\ast \to X$.
    Since $ \upU(X)_0 \simeq \ast $, \HTT{Proposition}{6.1.2.11} shows that the Čech nerve  $ \upU(X) $ is a group object of $\dX$.
    Since $ \upU(X)_1 \simeq \Omega X $, it follows
    that the loop functor $ \Omega \colon \dX_{\ast} \to \dX_{\ast} $ factors as the composite
    \begin{equation*}
        \begin{tikzcd}[sep=1.5em]
            \dX_{\ast} \arrow[r, "\upU"] & \Grp(\dX) \arrow[r] & \dX_{\ast}
        \end{tikzcd}
    \end{equation*}
    of the functor given by the assignment $ X \mapsto \upU(X) $ followed by the forgetful functor $ \Grp(\dX) \to \dX_{\ast} $.
\end{proof}

We leave the following Splitting Lemma as an amusing exercise for the reader.

\begin{lemma}[Splitting Lemma]\label{lem:split}
    Let $ \dX $ be an $\infty$-category with finite limits, let
    \begin{equation*}
        \begin{tikzcd}[sep=1.5em]
            A \arrow[r, "i"] & B \arrow[r, "p"] & C
        \end{tikzcd}
    \end{equation*}
    be a fiber sequence of group objects in $ \dX $, and write $ m \colon B \times B \to B $ for the multiplication on $ B $.
    For any section $ s \colon C \to B $ of $ p $ on the level of underlying pointed objects of $ \dX $, the composite
    \begin{equation*}
        \begin{tikzcd}[sep=1.5em]
            A \times C \arrow[r, "i \times s"] & B \times B \arrow[r, "m"] & B
        \end{tikzcd}
    \end{equation*}
    is an equivalence in $ \dX_{\ast} $.
\end{lemma}

We now prove the Fundamental Hilton--Milnor Splitting.

\begin{proof}[Proof of \cref{hilton-milnor}]
    By \Cref{lem:fibwedgetoprod,loop-monoid}, there is a fiber sequence
    \begin{equation}\label{diag:HMfib}
        \begin{tikzcd}[sep=1.5em]
            \Omega \fib(X\vee Y\to X\times Y)  \arrow[r] & \Omega (X\vee Y) \arrow[r] & \Omega X\times \Omega Y
        \end{tikzcd}
    \end{equation}
    of group objects of $ \dX $.
    Note that the map $ \Omega (X\vee Y) \to \Omega X\times \Omega Y $ has a section defined by the composite
    \begin{equation*}
        \begin{tikzcd}[sep=3.25em]
            \Omega X \times \Omega Y  \arrow[r, "\Omega i_1 \times \Omega i_2"] & \Omega (X\vee Y) \times \Omega(X\vee Y) \arrow[r, "m"] & \Omega(X\vee Y) \comma
        \end{tikzcd}
    \end{equation*}
    where $ i_1 \colon X \to X \wedgesum Y $ and $ i_2 \colon Y \to X \wedgesum Y $ are the coproduct insertions, and $ m $ is the multiplication coming from the group structure on $ \Omega(X\vee Y) $.
    By \Cref{lem:split} the fiber sequence \eqref{diag:HMfib} splits, so applying \Cref{lem:fibwedgetoprod} we see that there are equivalences
    \begin{align*}
        \Omega (X\vee Y) &\simeq \Omega X \times \Omega Y \times \Omega \fib(X\vee Y\to X\times Y) \\ 
        &\simeq \Omega X\times \Omega Y \times \Omega \Sigma (\Omega X\wedge \Omega Y) \period \qedhere
    \end{align*}
\end{proof}


\section{Connectedness \& the James Splitting}\label{sec:conncectednessssplit}

The purpose of this section is explain how to use a connectedness argument to show that if $ X $ is a pointed connected object of an $ \infty $-topos, then the comparison morphism 
\begin{equation*}
    c_{X} \colon \bigvee_{i = 1}^{\infty} \Sigma X^{\wedge i} \to \Sigma \Omega \Sigma X 
\end{equation*}
is an equivalence.
To do this, we start by reviewing the basics of connectedness in an $ \infty $-topos (\cref{subsec:connectedness}).
We also prove some basic connectedness results that we need in our proof of the metastable EHP sequence in \Cref{sec:EHP}.
\Cref{subsec:infsplitting} proves the infinite James and Hilton--Milnor Splittings for connected objects and explains how to deduce Wickelgren and William's James Splitting in motivic spaces over a perfect field from these results.


\subsection{Connectedness and the Blakers--Massey Theorem}\label{subsec:connectedness}

In this subsection, we review the basic properties of $ k $-truncated and $ k $-connected morphisms in an $ \infty $-topos.
We also recall the Blakers--Massey Theorem (\Cref{thm:MR4186137}) and Freudenthal Suspension Theorem (\Cref{cor:Freudenthal}) in an $ \infty $-topos, since our proof of the metastable EHP sequence in \cref{sec:EHP} relies on these results.

The reader interested in the details of the results reviewed here should consult \cites[\HTTsubsec{6.5.1}]{HTT}[\S3.3]{MR4186137} for connectedness results, and \cite{MR4186137} for the Blakers--Massey Theorem.

\begin{definition}\label{def:truncate}
    Let $ \dX $ be an $ \infty $-topos.
    For each integer $ k \geq -2 $, define \defn{$ k $-truncatedness} for morphisms in $ \dX $ recursively as follows.
    \begin{enumerate}[(\ref*{def:truncate}.1)]
        \item A morphism $ f $ is \defn{$ (-2) $-truncated} if $ f $ is an equivalence.

        \item For $ k \geq -1 $, a morphism $ f \colon X \to Y $ is \defn{$ k $-truncated} if the diagonal $ \Delta_f \colon X \to X \times_Y X $ is $ (k-1) $-truncated.
    \end{enumerate}
    An object $ X \in \dX $ is \defn{$ k $-truncated} if the unique morphism $ X \to \ast $ is $ k $-truncated.

    Write $ \dX_{\leq k} \subset \dX $ for the full subcategory spanned by the $ k $-truncated objects. 
    The inclusion $ \dX_{\leq k} \subset \dX $ admits a left adjoint which we denote by  $ \trun_{\leq k} \colon \dX \to \dX_{\leq k} $.
\end{definition}

\begin{example}
    Let $ \cC $ be a small $ \infty $-category equipped with a Grothendieck topology $ \tau $, and let $ k \geq -2 $ be an integer. 
    Then a sheaf $ \dF \in \Sh_{\tau}(\cC) $ of spaces on $ \cC $ with respect to $ \tau $ is $ k $-truncated if and only if $ \dF(c) $ is a $ k $-truncated space for every $ c \in \cC $.
    That is, $ \dF $ is $ k $-truncated if and only if $ \dF $ is a sheaf of $ k $-truncated spaces.
\end{example}

\begin{remark}
    If $ \dX $ is an $ \infty $-topos, then the full subcategory $ \dX_{\leq 0} $ spanned by the $ 0 $-truncated objects is an ordinary topos, i.e., a category of sheaves of sets on a Grothendieck site.
\end{remark}

\begin{recollection}
    Let $ \dX $ be an $ \infty $-topos.
    A morphism $ f \colon X \to Y $ in $ \dX $ is an \defn{effective
    epimorphism} if the augmentation $ \Cech(f) \to Y $ exhibits $ Y $ as the colimit of the Čech nerve of $ f $ (see \cref{cech-defn}).
    Equivalently, $ f $ is an effective epimorphism if and only if $ \trun_{\leq 0}(f) \colon \trun_{\leq 0}(X) \to \trun_{\leq 0}(Y) $ is an effective epimorphism in the ordinary topos $ \dX_{\leq 0} $ of $ 0 $-truncated objects of $ \dX $ \HTT{Proposition}{7.2.1.14}. 
\end{recollection}

\begin{example}
    A morphism $ f \colon X \to Y $ in the $ \infty $-topos $ \Spc $ of spaces is an effective epimorphism if and only if $ \uppi_0(f) \colon \uppi_0(X) \to \uppi_0(Y) $ is a surjection of sets.
\end{example}

\begin{definition}\label{def:conn}
    Let $ \dX $ be an $ \infty $-topos.
    For each integer $ k \geq -2 $, define \defn{$ k $-connectedness} for morphisms in $ \dX $ recursively as follows.
    \begin{enumerate}[(\ref*{def:conn}.1)]
         \item Every morphism is \defn{$ (-2) $-connected}.

        \item For $ k \geq -1 $, a morphism $ f \colon X \to Y $ is \defn{$ k $-connected} if $ f $ is an effective epimorphism and the diagonal $ \Delta_f \colon X \to X \times_Y X $ is $ (k-1) $-connected.
    \end{enumerate}
    A morphism $ f $ is \defn{$ \infty $-connected} if $ f $ is $ k $-connected for all $ k \geq -2 $.
    
    An object $ X \in \dX $ is \defn{$ k $-connected} if the unique morphism $ X \to \ast $ is $ k $-connected.
\end{definition}

\begin{remark}
    A morphism $ f $ is $ (-1) $-connected if and only if $ f $ is an effective epimorphism.
\end{remark}

\begin{definition}\label{defn:hypercomplete}
    An $ \infty $-topos $ \dX $ is \defn{hypercomplete} if every $ \infty $-connected morphism in $ \dX $ is an equivalence.
\end{definition}

\begin{warning}\label{warning:connectedness}
    Our conventions for connectedness follow those of Anel, Biedermann, Finster, and Joyal \cite[\S3.3]{MR4186137}.
    For $ k \geq 0 $, a homotopy type $ X $ is $ k $-connected in our sense if and only if $ X $ is $ k $-connected in the classical terminology \cites[\S6.7]{MR2456045}[p. 346]{MR1867354}[Chapter 10, \S4]{MR1702278}.
    In particular, $ X $ is $ 0 $-connected if and only if $ X $ is path-connected.
    This convention \textit{differs} from the classical one for maps: an $ n $-connected map in our sense is an $ (n+1) $-connected map in the classical sense.

    Comparing to Lurie's terminology \cite[\HTTsubsec{6.5.1}]{HTT}, an object or morphism is $ n $-connect\textit{ed} in our sense if and only if it is $ (n+1) $-connect\textit{ive} in Lurie's sense.
    One of the benefits of our terminological choice is that the constant factors in many connectedness estimates are eliminated (see \Cref{thm:MR4186137,cor:Freudenthal}).
\end{warning}

The following basic properties of $ k $-connected morphisms are proven in \cite[\HTTsubsec{6.5.1}]{HTT}.

\begin{prop}\label{prop:connectednessfacts}
    Let $ \dX $ be an $ \infty $-topos and $ k \geq -2 $ be an integer.
    Then:
    \begin{enumerate}[{\upshape (\ref*{prop:connectednessfacts}.1)}]
        \item\label{prop:connectednessfacts.0} The class of $ k $-connected morphisms in $ \dX $ is stable under composition.

        \item\label{prop:connectednessfacts.1} The class of $ k $-connected morphisms in $ \dX $ is stable under pushout along any morphism.

        \item\label{prop:connectednessfacts.2} The class of $ k $-connected morphisms in $ \dX $ is stable under pullback along any morphism.

         \item\label{prop:connectednessfacts.5} The class of $ k $-connected objects in $ \dX $ is stable under finite products.

         \item\label{prop:connectednessfacts.7} Given morphisms $ f \colon X \to Y $ and $ g \colon Y \to Z $ in $ \dX $ where $ f $ is $ k $-connected, the morphism $ g $ is $ k $-connected if and only if $ gf $ is $ k $-connected. 

        \item\label{prop:connectednessfacts.3} Given a morphism $ f \colon X \to Y $ in $ \dX $ with a section $ s \colon Y \to X $, the morphism $ f $ is $ (k+1) $-connected if and only if the section $ s $ is $ k $-connected.

        \item\label{prop:connectednessfacts.4} An object $ X \in \dX $ is $ k $-connected if and only if the $ k $-truncation $ \trun_{\leq k}(X) $ of $ X $ is terminal in $ \dX $.
    \end{enumerate}
\end{prop}

Since the $ k $-truncation functor $ \trun_{\leq k} \colon \dX \to \dX $ preserves filtered colimits, from \enumref{prop:connectednessfacts}{4} we deduce:

\begin{corollary}\label{cor:filteredcolimkconn}
    Let $ \dX $ be an $ \infty $-topos and $ k \geq -2 $ be an integer.
    Then the class of $ k $-connected objects of $ \dX $ is stable under filtered colimits.
\end{corollary}

In the $ \infty $-topos of spaces, the following connectedness estimates are usually done by appealing to cell structures.
Such arguments are unavailable in an arbitrary $ \infty $-topos, so we deduce these connectedness estimates from \Cref{prop:connectednessfacts}.

\begin{prop}\label{prop:smashconn}
    Let $ \dX $ be an $ \infty $-topos, $ X,Y \in \dX_{\ast} $ pointed objects, and $ k,\el \geq 0 $ integers.
    If $ X $ is $ k $-connected and $ Y $ is $ \el $-connected, then:
    \begin{enumerate}[{\upshape (\ref*{prop:smashconn}.1)}]
        \item\label{prop:smashconn.1} The suspension $ \Sigma X $ is $ (k+1) $-connected.

        \item\label{prop:smashconn.2} The coproduct insertion $ X \to X \wedgesum Y $ is $ (\el-1) $-connected.

        \item\label{prop:smashconn.3} The natural morphism $ X \wedgesum Y \to X \times Y $ is $ (k + \el) $-connected.

        \item\label{prop:smashconn.4} The smash product $ X \smashprod Y $ is $ (k + \el + 1) $-connected.

        \item\label{prop:smashconn.5} For each postive integer $ n $, the $ n $-fold smash product $ X^{\smashprod n} $ is $ (n(k+1)-1) $-connected.
    \end{enumerate}
\end{prop}

\begin{proof}
    For \enumref{prop:smashconn}{1} note that since $ k $-connected morphisms are stable under pushout \enumref{prop:connectednessfacts}{1}, the definition of the suspension $ \Sigma X $ as a pushout shows that the basepoint $ \ast \to \Sigma X $ is $ k $-connected.
    Hence $ \Sigma X $ is $ (k + 1) $-connected \enumref{prop:connectednessfacts}{3}.

    Second, \enumref{prop:smashconn}{2} follows from the fact that the basepoint $ \ast \to Y $ is $ (\el-1) $-connected \enumref{prop:connectednessfacts}{3}, combined with the fact that $ (\el - 1) $-connected morphisms are stable under pushout \enumref{prop:connectednessfacts}{1}.
    Third, \enumref{prop:smashconn}{3} follows from the fact that the basepoints $ \ast \to X $ and $ \ast \to Y $ are $ (k-1) $-connected and $ (\el-1) $-connected \enumref{prop:connectednessfacts}{3}, respectively, and a general fact about pushout-products of connected morphisms \cite[Corollary 3.3.7(4)]{MR4186137}.

    Now we prove \enumref{prop:smashconn}{4}.
    Since $ (k+\el) $-connected morphisms are stable under pushout \enumref{prop:connectednessfacts}{1}, by \enumref{prop:smashconn}{1} and the pushout square   
    \begin{equation*}
        \begin{tikzcd}[sep=2em]
            X \wedgesum Y \arrow[r] \arrow[d] \arrow[dr, phantom, very near end, "\ulcorner", xshift=0.25em, yshift=-0.25em] & X \times Y \arrow[d] \\
            \ast \arrow[r] & X \smashprod Y
        \end{tikzcd}
    \end{equation*} 
    defining the smash product $ X \smashprod Y $, we see that the basepoint $ \ast \to X \smashprod Y $ is $ (k+\el) $-connected.
    Hence $ X \smashprod Y $ is $ (k+\el+1) $-connected \enumref{prop:connectednessfacts}{3}.

    Finally, \enumref{prop:smashconn}{5} follows from \enumref{prop:smashconn}{4} by induction.
\end{proof}

Now we record a convenient fact about the interaction between connectedness and pullbacks that we need in out proof of the metastable EHP sequence.

\begin{prop}\label{prop:cubeconn}
    Let $ \dX $ be an $ \infty $-topos, $ \el \geq -2 $ be an integer, and 
    \begin{equation*}
        \begin{tikzcd}[sep=2em]
            A \arrow[r, "f"] \arrow[d, "a"'] & C \arrow[d, "c" description] & B \arrow[d, "b"] \arrow[l, "g"'] \\
            A' \arrow[r, "f'"'] & C' & B' \arrow[l, "g'"]
        \end{tikzcd}
    \end{equation*} 
    be a commutative diagram in $ \dX $.
    If $ a $ and $ b $ are $ \el $-connected and $ c $ is $ (\el+1) $-connected, then the induced morphism on pullbacks $ a \times_c b \colon A \times_C B \to A' \times_{C'} B' $ is $ \el $-connected.
\end{prop}

\begin{proof}
    Since $ \el $-connected morphisms are stable under composition, by factoring the induced morphism $ A \times_C B \to A' \times_{C'} B' $ as a composite of induced morphisms
    \begin{equation*}
        \begin{tikzcd}[sep=3em]
            A \times_C B  \arrow[r, "a\, \times_{c} \id"] & A' \times_{C'} B \arrow[r, "\id \times_{\id} b"] & A' \times_{C'} B' \comma
        \end{tikzcd}
    \end{equation*}
    it suffices to prove the claim in the special case $ B = B' $ and the morphism $ b \colon B \to B' $ is the identity.
    To prove the claim when $ b $ is the identity, first write $ (\id_{A},f) \colon A \to A \times_{C'} C $ for the section of the projection $ A \times_{C'} C \to A $ induced by the commutative square
    \begin{equation*}
        \begin{tikzcd}[sep=2em]
            A \arrow[d, equals] \arrow[r, "f"] & C \arrow[d, "c"] \\
            A \arrow[r, "f'a"'] & C' \period
        \end{tikzcd}
    \end{equation*}
    Consider the following commutative diagram of pullback squares
    \begin{equation*}
        \begin{tikzcd}[sep=2em]
            A \times_C B \arrow[r] \arrow[d] \arrow[dr, phantom, very near start, "\lrcorner", xshift=-0.5em] & A \times_{C'} B \arrow[r] \arrow[d] \arrow[dr, phantom, very near start, "\lrcorner", xshift=-0.5em] & A' \times_{C'} B \arrow[r] \arrow[d] \arrow[dr, phantom, very near start, "\lrcorner", xshift=-0.5em] & B \arrow[d, "g"] \\
            A \arrow[r, "{(\id_{A},f)}"'] & A \times_{C'} C \arrow[r] \arrow[d] \arrow[dr, phantom, very near start, "\lrcorner", xshift=-0.5em] & A' \times_{C'} C \arrow[r] \arrow[d] \arrow[dr, phantom, very near start, "\lrcorner", xshift=-0.5em] & C \arrow[d, "c"] \\
            & A \arrow[r, "a"'] & A' \arrow[r, "f'"'] & C' \comma
        \end{tikzcd}
    \end{equation*}
    and notice that the composite middle horizontal map $ A \to C $ is $ f $.
    Since $ c \colon C \to C' $ is $ (\el + 1) $-connected, the projection $ A \times_{C'} C \to A $ is $ (\el + 1) $-connected \enumref{prop:connectednessfacts}{2}.
    Hence the section $ (\id_{A},f) \colon A \to A \times_{C'} C $ is $ \el $-connected \enumref{prop:connectednessfacts}{3}.
    Consequently, the induced morphism $ A \times_C B \to A \times_{C'} B $ is $ \el $-connected.
    Now note that since $ a \colon A \to A' $ is $ \el $-connected, the induced morphism $ A \times_{C'} B \to A' \times_{C'} B $ is $ \el $-connected.
    Hence the composite morphism $ A \times_C B \to A' \times_{C'} B $ is $ \el $-connected, as desired.
\end{proof}

\noindent In particular, \Cref{prop:cubeconn} shows that the class of $ \el $-connected morphisms in an $ \infty $-topos is closed under finite products.
Setting $ B = B' = \ast $ in \Cref{prop:cubeconn} we deduce:

\begin{corollary}\label{cor:fibconn}
    Let $ \dX $ be an $ \infty $-topos, $ \el \geq -2 $ be an integer, and  
    \begin{equation*}
        \begin{tikzcd}[sep=2em]
            A \arrow[r, "f"] \arrow[d, "a"'] & C \arrow[d, "c"] \\
            A' \arrow[r, "f'"'] & C'
        \end{tikzcd}
    \end{equation*} 
    be a commutative square in $ \dX $.
    If $ a $ is $ \el $-connected and $ c $ is $ (\el+1) $-connected, then for every point $ x \colon \ast \to C $, the induced morphism $ \fib_x(f) \to \fib_{cx}(f') $ on fibers is $ \el $-connected.
\end{corollary}

We conclude this subsection by recalling the Blakers--Massey and Freudenthal Suspension Theorems in the setting of $ \infty $-topoi.

\begin{theorem}[{Blakers--Massey \cite[Corollary 4.3.1]{MR4186137}}]\label{thm:MR4186137}
    Let $ \dX $ be an $ \infty $-topos and let 
    \begin{equation*}
        \begin{tikzcd}[sep=2em]
            A \arrow[r, "g"] \arrow[d, "f"'] \arrow[dr, phantom, very near end, "\ulcorner", xshift=0.25em, yshift=-0.25em] & C \arrow[d] \\
            B \arrow[r] & D
        \end{tikzcd}
    \end{equation*} 
    be a pushout square in $ \dX $.
    If $ f $ is $ k $-connected and $ g $ is $ \el $-connected, then the induced morphism $ A \to B \times_D C $ is $ (k + \el) $-connected.
\end{theorem}

\noindent As in the classical setting, applying the Blakers--Massey Theorem to the pushout defining the suspension immediately implies the Freudenthal Suspension Theorem.

\begin{corollary}[Freudenthal Suspension Theorem]\label{cor:Freudenthal}
    Let $ \dX $ be an $ \infty $-topos, and $ X \in \dX_{\ast} $ a pointed $ k $-connected object.
    Then the unit morphism $ X \to \Omega\Sigma X $ is $ 2k $-connected.
\end{corollary}


\subsection{The infinite James and Hilton--Milnor Splittings}\label{subsec:infsplitting}

We now explain how to use a connectedness argument to show that if $ X $ is a pointed $ 0 $-connected object of an $ \infty $-topos, then the comparison morphism 
\begin{equation*}
    c_{X} \colon \bigvee_{i = 1}^{\infty} \Sigma X^{\wedge i} \to \Sigma \Omega \Sigma X 
\end{equation*}
introduced in \Cref{ntn:Jamescomparison} is an equivalence (\Cref{prop:infiniteJamestopos}).
The James Splitting gives us access to generalized Hopf invariants in this very general setting, and implies the stable Snaith Splitting for $ \Omega \Sigma X $ \cite{MR339155}.
We also prove the infinite version of the Hilton--Milnor Splitting (\Cref{cor:infiniteHiltonMilnortopos}).
Using Morel's unstable $ \AA^1 $-connectivity Theorem, we explain why these infinite splittings hold in motivic spaces over a perfect field (\Cref{cor:infinitemotivicsplittings}).
As an application, we give a new description of \smash{$ \Omega\Sigma^2(\PP^1 \smallsetminus \{0,1,\infty\}) $} (\Cref{ex:OmegaSigma2P13pts}).

The key tool to prove all of these results is the following lemma about how infinite James Splittings interact with localizations.

\begin{lemma}\label{lem:localizationtrick}
    Let $ \dY $ be an $\infty$-category, $ j \colon \dX \hookrightarrow \dY $ a full subcategory, and assume that the inclusion $ j $ admits a left adjoint $ L \colon \dY \to \dX $.
    Assume that:
    \begin{enumerate}[{\upshape (\ref*{lem:localizationtrick}.1)}]
        \item The $ \infty $-categories $ \dX $ and $ \dY $ have universal pushouts and the $ \infty $-category $ \dY_{\ast} $ has countable coproducts.

        \item The functor $ L \colon \dY \to \dX $ commutes with finite products and the formation of loop objects (e.g. $ L $ is left exact). 
    \end{enumerate}
    If $ X \in \dX_{\ast} $ is a pointed object with the property that the comparison morphism 
    \begin{equation*}
        c_{j(X)} \colon \bigvee_{i \geq 1} \Sigma_{\dY} j(X)^{\wedge i} \to \Sigma_{\dY} \Omega_{\dY} \Sigma_{\dY} j(X) 
    \end{equation*}
    is an equivalence in $ \dY_{\ast} $, then the comparison morphism
    \begin{equation*}
        c_{X} \colon \bigvee_{i \geq 1} \Sigma_{\dX} X^{\wedge i} \to \Sigma_{\dX} \Omega_{\dX} \Sigma_{\dX} X 
    \end{equation*}
    is an equivalence in $ \dX_{\ast} $.
\end{lemma}

\begin{proof}
    Since the localization $ L \colon \dY \to \dX $ preserves the terminal object, the functor $ L $ descends to the level of pointed objects.
    Moreover, since $ L $ preserves finite products, the functor $ L \colon \dY_{\ast} \to \dX_{\ast} $ commutes with smash products.
    Since the functor $ L \colon \dY_{\ast} \to \dX_{\ast} $ also commutes with the formation of loop objects, we see that the morphism $ c_{X} $ is equivalent to $ L(c_{j(X)}) $.
    The assumption that the morphism $ c_{j(X)} $ is an equivalence completes the proof.
\end{proof}

\begin{prop}[James Splitting]\label{prop:infiniteJamestopos}
    Let $ \dX $ be an $ \infty $-topos.
    Then for each $ 0 $-connected pointed object $ X \in \dX_{\ast} $, the natural comparison morphism
    \begin{equation*}
        c_{X} \colon \bigvee_{i \geq 1} \Sigma X^{\wedge i} \to \Sigma \Omega \Sigma X 
    \end{equation*}
    is an equivalence in $ \dX_{\ast} $.
\end{prop}

\begin{proof}
    Since every $ \infty $-topos is a left exact localization of a presheaf $ \infty $-topos and presheaf $ \infty $-topoi are hypercomplete (see \Cref{defn:hypercomplete}), by \Cref{lem:localizationtrick} we are reduced to the case that $ \dX $ is hypercomplete.
    That is, it suffices to show that the morphism $ c_{X} $ is $ \infty $-connected.
    
    To see this, notice that for each integer $ n \geq 1 $, the summand inclusions of $ \bigvee_{i = 1}^n \Sigma X^{\smashprod i} $ into $ \bigvee_{i = 1}^{\infty} \Sigma X^{\smashprod i} $ and 
    \begin{equation*}
        \Sigma \Omega \Sigma X \simeq \paren{\bigvee_{1 \leq i \leq n} \Sigma X^{\wedge i}} \vee (X^{\smashprod n} \wedge \Sigma\Omega \Sigma X) 
    \end{equation*}
    (\Cref{jamesredux}) fit into a commutative triangle
    \begin{equation}\label{triang:comparison}
        \begin{tikzcd}[sep=2.5em]
            \bigvee_{i = 1}^n \Sigma X^{\smashprod i} \arrow[r] \arrow[dr] & \bigvee_{i = 1}^{\infty} \Sigma X^{\smashprod i} \arrow[d, "c_{X}"] \\
            & \Sigma \Omega \Sigma X \period
        \end{tikzcd}
    \end{equation}
    Since $ X $ is $ 0 $-connected, combining \enumref{prop:smashconn}{1}, \enumref{prop:smashconn}{4}, \enumref{prop:smashconn}{5} and \Cref{cor:Freudenthal} we see that $ X^{\smashprod n} \smashprod \Sigma \Omega \Sigma X $ is $ n $-connected.
    Hence \enumref{prop:smashconn}{2} shows that the summand inclusion
    \begin{equation*}
        \bigvee_{i = 1}^n \Sigma X^{\smashprod i} \to \Sigma \Omega \Sigma X
    \end{equation*}
    is $ (n-1) $-connected.
    Similarly, since $ X $ is $ 0 $-connected, combining \enumref{prop:smashconn}{1}, \enumref{prop:smashconn}{5}, and \Cref{cor:filteredcolimkconn} we see that the object $ \bigvee_{i \geq n + 1}^{\infty} \Sigma X^{\smashprod i} $ is $ (n + 1) $-connected.
    Again applying \enumref{prop:smashconn}{2}, we see that the summand inclusion
    \begin{equation*}
        \bigvee_{i = 1}^n \Sigma X^{\smashprod i} \to \bigvee_{i = 1}^{\infty} \Sigma X^{\smashprod i}
    \end{equation*}
    is $ n $-connected.
    The commutativity of the triangle \eqref{triang:comparison} combined with \enumref{prop:connectednessfacts}{7} show that the morphism
    \begin{equation*}
        c_{X} \colon \bigvee_{i \geq 1} \Sigma X^{\wedge i} \to \Sigma \Omega \Sigma X 
    \end{equation*}
    is $ (n - 1) $-connected.
    Since this is true for each integer $ n \geq 1 $, we have shown that morphism $ c_{X} $ is $ \infty $-connected, as desired.
\end{proof}

\begin{remark}
    Of course, in the proof of \Cref{prop:infiniteJamestopos} we can further reduce to the case $ \dX = \Spc $ and the claim follows from the classical James Splitting.
    The purpose of our proof is to provide an explaination that does not appeal to the classical result but, rather, only uses basic manipulations available in an $ \infty $-topos.
\end{remark}

\begin{prop}[Hilton--Milnor Splitting, general version]\label{prop:infiniteHiltonMilnorgeneral}
    Let $ \dX $ be an $ \infty $-category with universal pushouts and countable coproducts, and let $ X,Y \in \dX_{\ast} $ be pointed objects.
    \begin{enumerate}[{\upshape (\ref*{prop:infiniteHiltonMilnorgeneral}.1)}]
        \item\label{prop:infiniteHiltonMilnorgeneral.1} If the comparison morphism $ c_Y $ is an equivalence, then there is a natural equivalence of pointed objects
        \begin{equation*}
            \Omega X \times \Omega\Sigma Y \times \Omega\paren{\Sigma\Omega X \smashprod \bigvee_{j \geq 1} Y^{\wedge j}} \similarrightarrow \Omega(X \wedgesum \Sigma Y) \period
        \end{equation*}

        \item\label{prop:infiniteHiltonMilnorgeneral.2} If the comparison morphisms $ c_X $ and $ c_Y $ are equivalences, then there is a natural equivalence of pointed objects
        \begin{equation*}
            \Omega\Sigma X \times \Omega\Sigma Y \times \Omega\Sigma\paren{\bigvee_{i,j \geq 1} X^{\wedge i} \smashprod Y^{\wedge j}} \similarrightarrow \Omega\Sigma(X \wedgesum Y) \period
        \end{equation*}
    \end{enumerate}
\end{prop}

\begin{proof}
    For \enumref{prop:infiniteHiltonMilnorgeneral}{1}, first note that the Hilton--Milnor Splitting of \Cref{hilton-milnor} provides an equivalence
    \begin{equation*}
        \Omega(X \vee \Sigma Y) \simeq \Omega X\times \Omega \Sigma Y\times \Omega \Sigma(\Omega X \wedge \Omega \Sigma Y) \period
    \end{equation*}
    The claim now follows from the natural natural equivalences
    \begin{align*}
        \Omega\paren{\Sigma\Omega X \smashprod \bigvee_{j \geq 1} Y^{\wedge j}} &\simeq \Omega\paren{\Omega X \smashprod \bigvee_{j \geq 1} \Sigma Y^{\wedge j}} && (\text{\Cref{lem:suspensionsmash}}) \\ 
        &\similarrightarrow \Omega (\Omega X \wedge \Sigma\Omega\Sigma Y) && (c_Y \text{ is an equivalence})\\
        &\simeq \Omega \Sigma(\Omega X \wedge \Omega\Sigma Y) && (\text{\Cref{lem:suspensionsmash}})
    \end{align*}

    Similarly, \enumref{prop:infiniteHiltonMilnorgeneral}{2} follows from \Cref{lem:suspensionsmash}, \enumref{prop:infiniteHiltonMilnorgeneral}{1}, and the assumption that \smash{$ c_{X} \colon \bigvee_{i \geq 1} \Sigma X^{\wedge i} \to \Sigma \Omega \Sigma X $} is an equivalence
\end{proof}

\begin{example}[Hilton--Milnor Splitting]\label{cor:infiniteHiltonMilnortopos}
    Let $ \dX $ be an $ \infty $-topos and let $ X,Y \in \dX_{\ast} $ be pointed objects.
    By \Cref{prop:infiniteJamestopos}, the hypotheses of \enumref{prop:infiniteHiltonMilnorgeneral}{1} are satisfied if $ Y $ is $ 0 $-connected, and the hypotheses of \enumref{prop:infiniteHiltonMilnorgeneral}{2} are satisfied if both $ X $ and $ Y $ are $ 0 $-connected.
\end{example}

The next application is that the infinite James and Hilton--Milnor Splittings hold for \textit{$ \AA^1 $-$ 0 $-connected} motivic spaces over a perfect field.
The infinite James Splitting was first proven by Wickelgren and Williams using the \textit{James filtration} \cite[Theorem 1.5]{MR3981318}; the infinite Hilton--Milnor Splittings in this context is new.

\begin{recollection}
    Let $ S $ be a scheme and $ n \geq 0 $ an integer.
    A motivic space $ X \in \HH(S) \subset \Shnis(\Sm_S) $ is \defn{$ \AA^1 $-$ n $-connected} if $ X $ is an $ n $-connected object of the Nisnevich $ \infty $-topos $ \Shnis(\Sm_S) $.
\end{recollection}

\begin{corollary}\label{cor:infinitemotivicsplittings}
    Let $ k $ be a perfect field and let $ X $ and $ Y $ be pointed motivic spaces over $ k $.
    \begin{enumerate}[{\upshape (\ref*{cor:infinitemotivicsplittings}.1)}]
        \item\label{cor:infinitemotivicsplittings.1} If $ X $ is $ \AA^1 $-$ 0 $-connected, then the natural comparison morphism
        \begin{equation*}
            c_{X} \colon \bigvee_{i \geq 1} \Sigma X^{\wedge i} \to \Sigma \Omega \Sigma X 
        \end{equation*}
        is an equivalence in $ \HH(\Spec(k))_{\ast} $.

        \item\label{cor:infinitemotivicsplittings.2} If $ Y $ is $ \AA^1 $-$ 0 $-connected, then there is a natural equivalence of pointed motivic spaces
        \begin{equation*}
            \Omega X \times \Omega\Sigma Y \times \Omega\paren{\Sigma\Omega X \smashprod \bigvee_{j \geq 1} Y^{\wedge j}} \similarrightarrow \Omega(X \wedgesum \Sigma Y) \period
        \end{equation*}

        \item\label{cor:infinitemotivicsplittings.3} If $ X $ and $ Y $ are $ \AA^1 $-$ 0 $-connected, then there is a natural equivalence of pointed motivic spaces
        \begin{equation*}
            \Omega\Sigma X \times \Omega\Sigma Y \times \Omega\Sigma\paren{\bigvee_{i,j \geq 1} X^{\wedge i} \smashprod Y^{\wedge j}} \similarrightarrow \Omega\Sigma(X \wedgesum Y) \period
        \end{equation*}
    \end{enumerate}
\end{corollary}

\begin{proof}
    First note that \enumref{cor:infinitemotivicsplittings}{2} and \enumref{cor:infinitemotivicsplittings}{3} follow from \Cref{prop:infiniteHiltonMilnorgeneral} and \enumref{cor:infinitemotivicsplittings}{1}.
    To verify \enumref{cor:infinitemotivicsplittings}{1}, note that since pushouts are universal in the $ \infty $-category $ \HH(\Spec(k)) $ of motivic spaces over $ k $ (\Cref{ex:motivicspaces}) and $ \HH(\Spec(k)) $ is a localization of the $ \infty $-topos $ \Shnis(\Sm_k) $ of Nisnevich sheaves on $ \Sm_k $, it suffices to check that the motivic localization functor
    \begin{equation*}
        \Lmot \colon \Shnis(\Sm_k) \to \HH(\Spec(k))
    \end{equation*}
    satisfies the hypotheses of \Cref{lem:localizationtrick}.
    To see this, first note that over any base scheme the motivic localization commutes with finite products \cite[Corollary 3.5]{MR3570135}.
    Second, since the field $ k $ is perfect, Morel's unstable $ \AA^1 $-connectivity Theorem \cite[Theorems 5.46 and 6.1]{MR2934577} implies that motivic localization commutes with the formation of loop objects \cites[Theorem 2.4.1]{MR3654105}[Theorem 6.46]{MR2934577}.
\end{proof}

\begin{example}\label{ex:OmegaSigma2P13pts}
    Let $ k $ be a perfect field.
    \Cref{cor:infinitemotivicsplittings} gives the following variant on \Cref{example:P1threepts}.
    Since $ \Sigma\GGm \equivalent \PP^1 $ is $ \AA^1 $-$ 0 $-connected, \enumref{cor:infinitemotivicsplittings}{3} provides equivalences of pointed motivic spaces over $ k $
    \begin{align*}
        \Omega\Sigma^2(\PP_k^1 \smallsetminus \{0,1,\infty\}) &\equivalent \Omega\Sigma\big(\PP_k^1 \wedgesum \PP_k^1\big) \\ 
        &\equivalent \Omega\Sigma\PP_k^1 \times \Omega\Sigma\PP_k^1 \times \Omega\Sigma\paren{\bigvee_{i,j \geq 1} \big(\PP_k^1\big)^{\wedge (i+j)}} \\ 
        &\equivalent \Omega\Sph{3,1} \times \Omega\Sph{3,1} \times \Omega\paren{\bigvee_{i,j \geq 1} \Sph{2(i+j)+1,i+j} } \period
    \end{align*}
\end{example}

We complete this section by constructing the \textit{Hopf maps} that appear in the metastable EHP sequence.

\begin{construction}[Hopf maps]\label{hopf}
    Let $ \dX $ be an $ \infty $-topos and $ X $ a pointed $ 0 $-connected object of $ \dX $.
    For each integer $ n \geq 1 $, we define the \defn{Hopf map} $ \Hopf_n \colon \Omega \Sigma
    X\to \Omega \Sigma X^{\wedge n} $ as the adjoint to the collapse map
    \begin{equation*}
        \Sigma \Omega \Sigma X \simeq \bigvee_{i\geq 1} \Sigma X^{\wedge i} \to \Sigma X^{\smashprod n}
    \end{equation*}
    induced by the James Splitting of \Cref{prop:infiniteJamestopos}.
\end{construction}


\section{The James filtration \& metastable EHP sequence}\label{sec:EHP}

In classical algebraic topology, the \textit{metastable EHP sequence} is the statement that the composite
\begin{equation*}
    \begin{tikzcd}[sep=1.5em]
        X \arrow[r] & \Omega \Sigma X \arrow[r, "\Hopf_2"] & \Omega \Sigma X^{\wedge 2}
    \end{tikzcd}
\end{equation*}
is a fiber sequence in a range depending on the connectedness of $ X $, known as the \textit{metastable range}.
Here the first map $ X \to \Omega\Sigma X $ is the unit and $ \Hopf_2 $ is the Hopf map (\Cref{hopf}).
For the higher Hopf maps $ \Hopf_n \colon \Omega\Sigma X \to \Omega \Sigma X^{\wedge n} $, there is an analogous fiber sequence in a range
\begin{equation*}
    \begin{tikzcd}[sep=1.5em]
        \J_{n-1}(X) \arrow[r] & \Omega \Sigma X \arrow[r, "\Hopf_n"] & \Omega \Sigma X^{\wedge n} \comma
    \end{tikzcd}
\end{equation*}
where $ \J_{n-1}(X) $ is the $ (n-1) $\textsuperscript{st} piece of the \textit{James filtration} on $ \Omega\Sigma X $.

This section is dedicated to a non-computational proof of the metastable EHP sequence in an $ \infty $-topos that only makes use of the Blakers--Massey Theorem and some basic connectedness results (see \Cref{metastable}).
In \cref{subsec:Jamesconstruction} we review the James filtration.
In \cref{subsec:splitJamesconstr} we refine the James Splitting to a splitting
\begin{equation*}
    \Sigma \J_n(X) \simeq \bigvee_{i=1}^{n} \Sigma X^{\wedge i} \period
\end{equation*}
In \cref{subsec:EHPproofs}, we give our non-computational proof of the metastable EHP sequence via the Blakers--Massey Theorem, and also record a computational proof for posterity.


\subsection{The James filtration}\label{subsec:Jamesconstruction}

Classically, the James filtration $ \{\J_n(X)\}_{n\geq 0} $ provides a multiplicative filtration
on the free monoid $ \J(X) $ on a pointed space $ X $, in the homotopical sense.
At the point-set level, $ \J(X) $ can be presented as the free topological monoid on $ X $, and $ \J_n(X) $ can be identified the subspace of words of length at most $ n $ in $ \J(X) $.
Concatenation of words then supplies $\{\J_n(X)\}_{n\geq 0}$ with the structure of a filtered monoid.
Since the trivial monoid and trivial group coincide, if $ X $ is connected, then the free monoid $ \J(X) $ on $ X $ coincides with the free group $ \Omega\Sigma X $ on $ X $.

In a general $\infty$-category, we can define the James filtration as follows.
This definition is provided in \cite[Section 3]{james-type-thy} in the context of homotopy type theory; the arguments made in \cite[Section 3]{james-type-thy} are formal and valid in any $\infty$-topos.

\begin{construction}[James filtration]\label{james-def}
    Let $\dX$ be an $\infty$-category with finite products and pushouts, and let $ X \in \dX_{\ast} $ be a pointed object.
    For each integer $ n \geq 0 $ we define a pointed object $ \J_{n}(X) \in \dX_{\ast} $ as well as morphisms
    \begin{equation*}
        i_{n} \colon \J_{n}(X) \to \J_{n+1}(X) \andeq \alpha_{n} \colon X \times \J_{n}(X) \to \J_{n+1}(X)
    \end{equation*}
    in $ \dX_{\ast} $ recursively as follows.
    \begin{enumerate}[(\ref*{james-def}.1)]
        \item\label{james-def.1} We define $ \J_0(X) \colonequals \ast $, $ \J_1(X) \colonequals X $, the morphism $ i_0 \colon \ast \to X $ is the basepoint, and the morphism $ \alpha_0 \colon X \times \ast \to X $ is the projection $ \pr_1 \colon X \times \mathbin{\ast} \equivto X $.

        \item\label{james-def.2} For $ n \geq 2 $, we define $ \J_{n}(X) $, $ i_{n-1} $, and $ \alpha_{n-1} $ by the pushout square 
        \begin{equation}\label{james-def-diagram}
            \begin{tikzcd}[sep=2em]
                \displaystyle X \times \J_{n-2}(X) \coproductlimits^{\J_{n-2}(X)} \J_{n-1}(X) \arrow[r] \arrow[d] \arrow[dr, phantom, very near end, "\ulcorner", xshift=1em, yshift=-0.25em] & \J_{n-1}(X) \arrow[d, "i_{n-1}"] \\
                X \times \J_{n-1}(X) \arrow[r, "\alpha_{n-1}"'] & \J_{n}(X) \comma
            \end{tikzcd}
        \end{equation}
	    where: the top horizontal morphism is induced by the universal property of the
        pushout by the commutative square
        \begin{equation*}
            \begin{tikzcd}[sep=2em]
                \J_{n-2}(X) \arrow[r, "i_{n-2}"] \arrow[d, "{(\ast,\id)}"'] &
                \J_{n-1}(X) \arrow[d, equals] \\
                X \times \J_{n-2}(X) \arrow[r, "\alpha_{n-2}"'] & \J_{n-1}(X) \comma
            \end{tikzcd}
        \end{equation*}   
        and the left vertical morphism is induced by the universal property of the pushout by the commutative square
	    \begin{equation*}
			\begin{tikzcd}[sep=2.5em]
			    \J_{n-2}(X) \arrow[r, "i_{n-2}"] \arrow[d, "{(\ast,\id)}"'] &
			    \J_{n-1}(X) \arrow[d, "{(\ast, \id)}"] \\
			    X \times \J_{n-2}(X) \arrow[r, "\id \times i_{n-2}"'] & X\times \J_{n-1}(X) \period
			\end{tikzcd}
	    \end{equation*}	    
    \end{enumerate}

    For each positive integer $ n $, define a morphism $ a_n \colon X^{\times n} \to \J_n(X) $ as the composite
    \begin{equation*}
        \begin{tikzcd}[sep=2.75em]
            X^{\times n} \equivalent X^{\times n-1} \times \J_1(X) \arrow[r, "\id \times \alpha_1"] & X^{\times n-2} \times \J_2(X) \arrow[r, "\id \times \alpha_2"] & \cdots \arrow[r] & X \times \J_{n-1}(X) \arrow[r, "\alpha_{n-1}"] & \J_n(X) \period
        \end{tikzcd}
    \end{equation*}  
    Finally, define $ \J(X) \colonequals \colim_{n \geq 0} \J_n(X) $.
\end{construction}

\begin{definition}
    Let $\dX$ be an $\infty$-category with finite products and pushouts, and let $ X \in \dX_{\ast} $ be a pointed object.
    For each integer $ n \geq 0 $ define a morphism $ u_{n} \colon \J_{n}(X) \to \Omega\Sigma X $ recursively as follows.
    The morphism $ u_0 $ is the basepoint, and the morphism $ u_1 \colon X \to \Omega\Sigma X $ is the unit.
    For $ n \geq 2 $, the morphism $ u_{n} $ is induced by the commutative square
    \begin{equation*}
        \begin{tikzcd}[sep=2em]
            \displaystyle X \times \J_{n-2}(X) \coproductlimits^{\J_{n-2}(X)} \J_{n-1}(X) \arrow[r] \arrow[d] & \J_{n-1}(X) \arrow[d, "u_{n-1}"] \\
            X \times \J_{n-1}(X) \arrow[r, "m \circ {(u_1 \times u_{n-1})}"'] & \Omega \Sigma X
        \end{tikzcd}
    \end{equation*}
    where $ m \colon \Omega \Sigma X \times \Omega \Sigma X \to \Omega \Sigma X $ is the group multiplication.

    The morphisms $ u_n $ induce a morphism $ u \colon \J(X) \to \Omega \Sigma X $.
\end{definition}

\begin{theorem}[{\cite[Section 6]{james-type-thy}}]
    Let $ \dX $ be an $ \infty $-topos and $ X \in \dX_{\ast} $ a pointed object.
    If $ X $ is $ 0 $-connected, then the morphism $ u \colon \J(X) \to \Omega \Sigma X $ is an equivalence.
\end{theorem}

Brunerie \cite{james-type-thy} gives an elementary proof of the following connectedness estimate.

\begin{lemma}[{\cite[Proposition 4]{james-type-thy}}]\label{lem:inconnectedness}
    Let $ \dX $ be an $ \infty $-topos, $ k \geq 0 $ an integer, and $ X \in \dX_{\ast} $ a pointed $ k $-connected object.
    Then the morphism $ i_{n-1} \colon
    \J_{n-1}(X) \to \J_n(X) $ is $ (n(k+1) - 2) $-connected.
\end{lemma}

\begin{corollary}\label{cor:Jnconn}
    Let $ \dX $ be an $ \infty $-topos, $ k \geq 0 $ an integer, and $ X \in \dX_{\ast} $ a pointed $ k $-connected object.
    Then for each integer $ n \geq 0 $, the object $ \J_n(X) $ is $ k $-connected.
\end{corollary}

\begin{proof}
    If $ n = 0 $, then the claim is clear since $ \J_0(X) = \ast $.
    If $ n > 0 $, then by \Cref{lem:inconnectedness} the morphisms $ i_0,\ldots,i_{n-1} $ are all $ (k-1) $-connected.
    Hence the basepoint
    \begin{equation*}
        i_{n-1} \cdots i_0 \colon \ast \to \J_{n}(X) 
    \end{equation*}
    is $ (k-1) $-connected; equivalently, $ \J_n(X) $ is $ k $-connected \enumref{prop:connectednessfacts}{3}.
\end{proof}

\begin{lemma}[{\cite[Proposition 6]{james-type-thy}}]\label{lem:unconn}
    Let $ \dX $ be an $ \infty $-topos, $ k \geq 0 $ an integer, and $ X \in \dX_{\ast} $ a pointed $ k $-connected object.
    Then the morphism $ u_{n} \colon \J_{n}(X) \to \Omega\Sigma X $ is $ ((n+1)(k+1) - 2) $-connected.
\end{lemma}


\subsection{Splitting the James filtration}\label{subsec:splitJamesconstr}

The purpose of this subsection is to prove the following splitting of the James filtration, which we use in our proof of the metastable EHP sequence (\Cref{metastable}).

\begin{prop}\label{partial-splitting}
    Let $\dX$ be an $\infty$-category with universal pushouts, and let $ X\in \dX_{\ast}$.
    Then there is a splitting
    \begin{equation}\label{partial-expression}
        \Sigma \J_n(X) \simeq \bigvee_{1\leq i \leq n} \Sigma X^{\wedge i} \period
    \end{equation}
    If $ \dX $ is an $ \infty $-topos and $ X $ is $ 0 $-connected, then under the map $ \Sigma u_n \colon \Sigma \J_n(X)\to \Sigma \Omega \Sigma X $, the splitting \eqref{partial-expression}
    is an equivalence onto the first $n$ factors of the splitting $ \Sigma \Omega \Sigma X \equivalent \bigvee_{i \geq 1} \Sigma X^{\wedge i} $ of \Cref{prop:infiniteJamestopos}.
\end{prop}

The proof of \Cref{partial-splitting} requires some preliminaries.
We need to relate the cofiber of $ i_n $ to smash powers of $ X $; before doing so we need some preparatory lemmas.

\begin{lemma}\label{cofiber-product-inclusion}
    Let $\dX$ be an $\infty$-category with finite products and pushouts, and let $X,Y \in \dX_{\ast}$.
    Then there is a cofiber sequence
    \begin{equation*}
        \begin{tikzcd}[sep=1.5em]
            Y \arrow[r] & \cofib((\id_X, \ast) \colon X\to X\times Y) \arrow[r] & X\wedge Y \period 
        \end{tikzcd}
    \end{equation*}
\end{lemma}

\begin{proof}
    There is a map of cofiber sequences
    \begin{equation}\label{diag:cofibprod}
        \begin{tikzcd}[sep=2em]
	       X \arrow[r, equals] \arrow[d] & X \arrow[d, "{(\id_X, \ast)}"] \arrow[r] & \ast \arrow[d] \\
            X \vee Y \arrow[r] & X \times Y \arrow[r] & X\wedge Y \comma
        \end{tikzcd}
    \end{equation}
    where the leftmost vertical map is the coproduct insertion.
    The cofiber of the coproduct insertion $X\to X\vee Y$ is $Y$, and the cofiber of the basepoint $\ast\to X\wedge Y $ is $ X \wedge Y $.
    To conclude, note that taking vertical cofibers in \eqref{diag:cofibprod} results in a cofiber sequence.
\end{proof}

The following is a straightforward application of \Cref{lem:3by3}.

\begin{lemma}\label{total-cofiber}
    Let $\dX$ be an $\infty$-category with pushouts and a terminal object and let
    \begin{equation*}
        \begin{tikzcd}[sep=2em]
	       A \arrow[r] \arrow[d] & B \arrow[d]\\
           C \arrow[r] & D
        \end{tikzcd}
    \end{equation*}
    be a commutative square in $ \dX_{\ast} $.
    Then there is a natural equivalence
    \begin{equation*}
        \cofib\big(B \coproduct^A C \to D \big) \equivalent \cofib\big(\cofib(A\to C) \to \cofib(B \to D)\big) \period
    \end{equation*}
\end{lemma}

We are now ready to show that $ \cofib(i_n) \equivalent X^{\smashprod n+1} $.

\begin{prop}\label{Jamescofib}
    Let $\dX$ be an $\infty$-category with universal pushouts, and let $ X\in \dX_{\ast} $.
    Then for each integer $ n \geq 0 $, there is a natural equivalence
    \begin{equation*}
        \cofib(i_n \colon \J_n(X) \to \J_{n+1}(X)) \equivalent X^{\smashprod n+1} \period
    \end{equation*}
    Moreover, the composite
    \begin{equation*}
        \begin{tikzcd}[sep=1.5em]
            X^{\times n+1} \arrow[r, "a_n"] & \J_{n+1}(X) \arrow[r] & X^{\smashprod n+1} 
        \end{tikzcd}
    \end{equation*}
    is equivalent to the canonical map $ X^{\times n+1} \to X^{\smashprod n+1} $.
\end{prop}

\begin{proof}
    We prove the claim by induction on $ n $.
    For the base case, note that since the morphism $ i_0 $ is the basepoint $ \ast \to X $, the cofiber of $ i_0 $ is $ X $.
    For the inductive step we assume that $ \cofib(i_n) \equivalent X^{\smashprod n + 1} $ and show that $ \cofib(i_{n+1}) \equivalent X^{\smashprod n + 2} $.
    From the defining pushout square \eqref{james-def-diagram}, we see that 
    \begin{equation*}
        \cofib(i_{n+1}) \equivalent \cofib\paren{X\times \J_n(X) \coproduct^{\J_n(X)} \J_{n+1}(X)\to X\times \J_{n+1}(X)} \period
    \end{equation*}
    Applying \Cref{total-cofiber} shows that
    \begin{equation*}
        \cofib(i_{n+1}) \simeq \cofib\Big(\cofib(\J_n(X) \to X\times \J_n(X)) \to \cofib(\J_{n+1}(X) \to X\times \J_{n+1}(X))\Big) \comma
    \end{equation*}
    where the map of cofibers is induced by the map $ i_n \colon \J_n(X) \to \J_{n+1}(X) $. 
    By \cref{cofiber-product-inclusion}, there is a cofiber sequence
    \begin{equation*}
        X \to \cofib(\J_n(X) \to X\times \J_n(X)) \to X\wedge \J_n(X) \rlap{\ ;}
    \end{equation*}
    moreover, the map $ i_n \colon \J_n(X) \to \J_{n+1}(X) $ induces a map of cofiber sequences
    \begin{equation}\label{diag:cofibmapjames}
        \begin{tikzcd}[sep=2em]
    	    X \arrow[d, equals] \arrow[r] & \cofib(\J_n(X) \to X\times
    	    \J_n(X)) \arrow[d] \arrow[r] & X\wedge \J_n(X) \arrow[d, "\id_{X} \smashprod i_n"] \\
    	    X \arrow[r] & \cofib(\J_{n+1}(X) \to X\times \J_{n+1}(X)) \arrow[r] &
    	    X\wedge \J_{n+1}(X) \period
        \end{tikzcd}
    \end{equation}
    Since the leftmost vertical map is the identity, taking vertical cofibers in the map of cofiber sequences \eqref{diag:cofibmapjames} produces an equivalence between the vertical cofibers of the middle and right vertical maps.
    Since the cofiber of the middle vertical map is $ \cofib(i_{n+1}) $, we find that
    \begin{equation*}
        \cofib(i_{n+1}) \simeq \cofib\big(\id_{X} \smashprod i_n \colon X \wedge \J_n(X) \to X\wedge \J_{n+1}(X)\big) \period
    \end{equation*}
    Since pushouts in $ \dX $ are universal we have a natural equivalence
    \begin{equation*}
        \cofib(\id_{X} \smashprod i_{n} \colon X \wedge \J_n(X) \to X\wedge \J_{n+1}(X)) \simeq X \wedge \cofib(i_{n} \colon \J_n(X) \to \J_{n+1}(X))
    \end{equation*}
    By the inductive hypothesis, $\cofib(i_n) \simeq X^{\wedge n+1}$, so $ \cofib(i_{n+1}) \simeq  X^{\wedge n+2}$, as desired.
\end{proof}

Next we split the term 
\begin{equation*}
    \displaystyle \Sigma\Big(X \times \J_{n-1}(X) \coproductlimits^{\J_{n-1}(X)} \J_{n}(X)\Big)
\end{equation*}
in the pushout square \eqref{james-def-diagram} defining $ \Sigma \J_{n+1}(X) $ and prove \Cref{partial-splitting}.

\begin{lemma}\label{lem:suspendcorner}
    Let $ \dX $ be an $ \infty $-category with universal pushouts, $ X \in \dX_{\ast} $, and $ n \geq 1 $ an integer.
    Then there is a natural equivalence
    \begin{equation*}
        \displaystyle \Sigma\Big(X \times \J_{n-1}(X) \coproductlimits^{\J_{n-1}(X)} \J_{n}(X)\Big) \equivalent \Sigma(X \smashprod \J_{n-1}(X)) \wedgesum \Sigma X \wedgesum \Sigma \J_{n}(X) \period
    \end{equation*}
\end{lemma}

\begin{proof}
    Since suspension preserves pushouts, we have a pushout square
    \begin{equation}\label{diag:susp-corner}
        \begin{tikzcd}[sep=2em]
            \Sigma \J_{n-1}(X) \arrow[r] \arrow[d] \arrow[dr, phantom, very near end, "\ulcorner", xshift=2em, yshift=-0.25em] & \Sigma\J_{n}(X) \arrow[d] \\
            \Sigma(X \times \J_{n-1}(X)) \arrow[r] & \Sigma\Big(X \times \J_{n-1}(X) \coproductlimits^{\J_{n-1}(X)} \J_{n}(X)\Big) \period
        \end{tikzcd}
    \end{equation}
    Under the equivalence 
    \begin{equation*}
        \Sigma(X \times \J_{n-1}(X)) \simeq \Sigma (X\wedge \J_{n-1}(X)) \vee \Sigma X \vee \Sigma \J_{n-1}(X)
    \end{equation*}
    of \Cref{cofiber-product}, the left vertical map in \eqref{diag:susp-corner} is the coproduct insertion.
    Hence on pushouts we see that
    \begin{equation*}
        \Sigma\Big(X \times \J_{n-1}(X) \coproductlimits^{\J_{n-1}(X)} \J_{n}(X)\Big) \simeq \Sigma (X\wedge \J_{n-1}(X)) \vee \Sigma X \vee \Sigma \J_{n}(X) \period \qedhere
    \end{equation*}
\end{proof}

\begin{proof}[Proof of \Cref{partial-splitting}]
    We prove the claim by induction on $ n $.
    The base case where $ n = 1 $ is obvious.
    For the inductive step, assume that $ n \geq 1 $ and $ \Sigma \J_n(X) \equivalent \bigvee_{i=1}^n \Sigma X^{\smashprod n} $.
    By \Cref{Jamescofib} we have a cofiber sequence
    \begin{equation*}
        \begin{tikzcd}[sep=1.5em]
            \J_{n}(X) \arrow[r, "i_{n}"] & \J_{n+1}(X) \arrow[r] & X^{\wedge n+1} \comma
        \end{tikzcd}
    \end{equation*}
    so the inductive hypothesis and the duals of \Cref{loop-monoid,lem:split}, it suffices to define a retraction
    \begin{equation*}
        r_n \colon \Sigma \J_{n+1}(X) \to \Sigma \J_{n}(X)
    \end{equation*}
    of the map $ \Sigma i_{n} $.

    We construct the retractions $ r_n \colon \Sigma \J_{n+1}(X) \to \Sigma \J_{n}(X) $ inductively.
    For the base case, the retraction $ r_0 \colon \Sigma X \to \ast $ of $ \Sigma i_0 $ is the unique morphism.
    For the inductive step, assume that $ n \geq 1 $ and we have constructed a retraction $ r_{n-1}\colon\Sigma \J_n(X) \to \Sigma \J_{n-1}(X) $ of $ \Sigma i_{n-1} $; we use this to construct a retraction $ r_n $ of $ \Sigma i_n $.
    Since suspension preserves pushouts, suspending the defining pushout square \eqref{james-def-diagram} yields a pushout square
    \begin{equation}\label{diag:susp-james-diagram}
        \begin{tikzcd}[sep=2em]
    		\displaystyle \Sigma\Big(X \times \J_{n-1}(X) \coproductlimits^{\J_{n-1}(X)} \J_{n}(X)\Big) \arrow[r] \arrow[d] \arrow[dr, phantom, very near end, "\ulcorner", xshift=1em, yshift=-0.25em] & \Sigma \J_{n}(X) \arrow[d, "\Sigma i_{n}"] \\
    		\Sigma(X \times \J_{n}(X)) \arrow[r, "\Sigma \alpha_{n}"'] & \Sigma \J_{n+1}(X) \period
        \end{tikzcd}
    \end{equation}
    In order to define a retraction of $\Sigma i_n$, it suffices to define a retraction of the left vertical map in \eqref{diag:susp-james-diagram}, i.e., it suffices to define a retraction
    \begin{equation*}
        \Sigma(X\times \J_n(X)) \to \Sigma\Big(X \times \J_{n-1}(X) \coproductlimits^{\J_{n-1}(X)} \J_{n}(X)\Big) \period
    \end{equation*}
    
    By \cref{cofiber-product,lem:suspendcorner}, we have equivalences
    \begin{align*}
        \Sigma(X \times \J_{n}(X)) &\simeq \Sigma (X\wedge \J_{n}(X)) \vee \Sigma X \vee \Sigma \J_{n}(X) \\ 
        \intertext{and}
         \Sigma\Big(X \times \J_{n-1}(X) \coproductlimits^{\J_{n-1}(X)} \J_{n}(X)\Big) &\simeq \Sigma (X\wedge \J_{n-1}(X)) \vee \Sigma X \vee \Sigma \J_{n}(X) \period
    \end{align*}
    Moreover, the left vertical map in \eqref{diag:susp-james-diagram} is induced by the suspensions of the identity on $ X $, identity on $ \J_{n}(X) $, and the map $ i_n \colon \J_{n-1}(X) \to \J_{n}(X) $.
    Under the identifications 
    \begin{equation*}
        \Sigma (X\wedge \J_{n-1}(X)) \equivalent  X \wedge \Sigma \J_{n-1}(X) \andeq \Sigma (X\wedge \J_{n}(X)) \equivalent  X \wedge \Sigma \J_{n}(X)
    \end{equation*}
    of \Cref{lem:suspensionsmash}, we see that the map
    \begin{equation*}
        \id_{X} \smashprod r_{n-1} \colon \Sigma (X\wedge \J_{n-1}(X)) \equivalent  X \wedge \Sigma \J_{n-1}(X) \longrightarrow X \wedge \Sigma \J_{n}(X) \equivalent \Sigma (X\wedge \J_{n}(X))
    \end{equation*}
    is a retraction of $ \Sigma(\id_{X} \smashprod i_{n-1}) $.
    Hence the map 
    \begin{equation*}
        (\id_{X} \smashprod r_{n-1}) \vee \id \vee \id \colon \Sigma (X\wedge \J_{n}(X)) \vee \Sigma X \vee \Sigma \J_{n}(X) \longrightarrow \Sigma (X\wedge \J_{n-1}(X)) \vee \Sigma X \vee \Sigma \J_{n}(X)
    \end{equation*}
    supplies the desired retraction of the left vertical map in \eqref{diag:susp-james-diagram}.
\end{proof}

\subsection{Proofs of the metastable EHP sequence}\label{subsec:EHPproofs}

In this subsection, we present two proofs of the metastable EHP sequence in the setting of $ \infty $-topoi.
Before making a precise statement of the main result, we need the following easy lemma.

\begin{lemma}\label{easy-lemma}
    Let $\dX$ be an $\infty$-topos, $X\in \dX_{\ast}$ be a pointed $ 0 $-connected object, and $ n \geq 1 $ an integer.
    Then the composite
    \begin{equation*}
        \begin{tikzcd}[sep=1.5em]
            \J_{n-1}(X) \arrow[r, "u_n"] & \Omega \Sigma X \arrow[r, "\Hopf_n"] & \Omega \Sigma X^{\wedge n}
        \end{tikzcd}
    \end{equation*}
    is null. 
    Here $ \Hopf_n $ is the Hopf map of \Cref{hopf}.
\end{lemma}

\begin{proof}
    It suffices to prove the corresponding statement on adjoints: in other words, we need to
    show that the composite
    \begin{equation*}
        \begin{tikzcd}[sep=1.5em]
            \Sigma \J_{n-1}(X) \arrow[r, "\Sigma u_n"] & \Sigma \Omega \Sigma X \arrow[r] & \Sigma X^{\wedge n}
        \end{tikzcd}
    \end{equation*}
    is null.
    This is an immediate consequence of \Cref{partial-splitting}.
\end{proof}

We can now state the metastable EHP sequence.

\begin{theorem}[metastable EHP sequence]\label{metastable}
    Let $ \dX $ be an $ \infty $-topos, $ k \geq 0 $ an integer, and $ X \in \dX_{\ast} $ a pointed $ k $-connected object.
    Then for every integer $ n \geq 1 $, the morphism $ \J_{n-1}(X) \to \fib(\Hopf_{n})$ induced by \Cref{easy-lemma} is $ ((n+1)(k+1) - 3)$-connected.
\end{theorem}

\begin{remark}
    The metastable EHP sequence for $ \infty $-topoi of hypersheaves on an ordinary site with enough points proven by Asok--Wickelgren--Williams \cite[Proposition 3.1.4]{MR3654105} is a special case of \Cref{metastable}.
\end{remark}


The first proof of \Cref{metastable} we present is \textit{internal} to $ \infty
$-topoi, and only uses basic facts about connectedness and the James construction, as well as the Blakers--Massey Theorem.
The second reduces to the $ \infty $-topos $ \Spc $ of spaces, then uses the homology Whitehead
Theorem and Serre spectral sequence to give a calculational proof of the metastable EHP sequence in the classical setting. 
Both perspectives are valuable, and we present the second here in part because the calculational proof of the metastable EHP sequence does not seem to be easy to locate in the literature.

\begin{proof}[Internal proof of \Cref{metastable}]
    First we show that it suffices to prove the claim where we replace $ \fib(\Hopf_n) $ by the fiber of the natural morphism $\J_{n}(X) \to X^{\smashprod n} $.
    Observe that we have a commutative square
    \begin{equation}\label{eq:squaretakefibsof}
        \begin{tikzcd}[sep=2em]
            \J_{n}(X) \arrow[r] \arrow[d, "u_n"'] & X^{\smashprod n} \arrow[d] \\
            \Omega\Sigma X \arrow[r, "\Hopf_n"'] & \Omega\Sigma X^{\smashprod n} \comma
        \end{tikzcd}
    \end{equation}
    where the right vertical morphism is the unit.
    Since $ X $ is $ k $-connected, the morphism
    \begin{equation*}
        u_n \colon \J_{n}(X) \to \Omega\Sigma X
    \end{equation*}
    is $ ((n+1)(k+1)-2) $-connected (\Cref{lem:unconn}) and $ X^{\smashprod n} $ is $ (n(k+1)-1) $-connected \enumref{prop:smashconn}{5}.
    By the Freudenthal Suspension Theorem (\Cref{cor:Freudenthal}) the unit morphism $ X^{\smashprod n} \to \Omega\Sigma X^{\smashprod n} $ is $ 2(n(k+1)-1) $-connected.
    Since $ n \geq 1 $, we have that 
    \begin{equation*}
        2(n(k+1)-1) \geq (n+1)(k+1)-2 \comma
    \end{equation*}
    so that both of the vertical morphisms in \eqref{eq:squaretakefibsof} are $ ((n+1)(k+1)-2) $-connected.
    Applying \Cref{cor:fibconn} to the square \eqref{eq:squaretakefibsof}, we see that the induced morphism on horizontal fibers
    \begin{equation*}
        \fib(\J_n(X) \to X^{\smashprod n}) \to \fib(\Hopf_n)
    \end{equation*}
    is $ ((n+1)(k+1)-3) $-connected.
    Therefore, to prove that the morphism $ \J_{n-1}(X) \to \fib(\Hopf_n) $ is $ ((n+1)(k+1)-3) $-connected, it suffices to show that the induced morphism
    \begin{equation}\label{eq:wantconnected}
        \J_{n-1}(X) \to \fib(\J_n(X) \to X^{\smashprod n})
    \end{equation}
    is $ ((n+1)(k+1)-3) $-connected.

    Since $ X $ is $ k $-connected, $ \J_{n-1}(X) $ is $ k $-connected (\Cref{cor:Jnconn}) and the morphism
    \begin{equation*}
        i_{n-1} \colon \J_{n-1}(X) \to \J_n(X) 
    \end{equation*}
    is $ (n(k+1) - 2) $-connected (\Cref{lem:inconnectedness}). 
    Applying the Blakers--Massey Theorem (\Cref{thm:MR4186137}) to the cofiber sequence
    \begin{equation*}
        \begin{tikzcd}[sep=2em]
            \J_{n-1}(X) \arrow[r, "i_{n-1}"] \arrow[d] \arrow[dr, phantom, very near end, "\ulcorner", xshift=0.75em, yshift=-0.5em]  & \J_n(X) \arrow[d] \\
            \ast \arrow[r] & X^{\smashprod n}  
        \end{tikzcd}
    \end{equation*}
    provided by \cref{Jamescofib}, we see that morphism \eqref{eq:wantconnected} is $ ((n+1)(k+1)-3) $-connected.
\end{proof}

\begin{proof}[Computational proof of \Cref{metastable}]
    Let $ m \colonequals (n+1)(k+1) - 3 $; we need to show that the map
    \begin{equation*}
        \J_{n-1}(X)\to \fib(\Hopf_n)
    \end{equation*}
    is $ m $-connected.
    The following two facts allow us to reduce to proving the claim in the case that $ \dX = \Spc $.
    \begin{enumerate}[(1)]
        \item If the conclusion of \Cref{metastable} holds for the $ \infty $-topos $ \Spc $, then it holds for any presheaf $ \infty $-topos.

        \item If $ L \colon \dY \to \dX $ is a left exact left adjoint between $ \infty $-topoi, then $ L $ preserves: suspensions, loops, smash products, fibers, the James filtration, and $ m $-connectedness (this last fact is \HTT{Proposition}{6.5.1.16}).
    \end{enumerate}

    Now we prove the claim for $ \dX = \Spc $.
    The claim is trivial if $ n = 1 $, so assume that $ n \geq 2 $.
    Since $X$ is $k$-connected by assumption, the smash power \smash{$X^{\wedge n}$} is $ (nk+n-1) $-connected \enumref{prop:smashconn}{5}.
    Since \smash{$\Omega \Sigma X^{\wedge n}$} is simply-connected, the Serre spectral sequence for (integral) homology has \smash{$\E^2$}-page
    \begin{equation*}
        \E^2_{p,q} = \H_p(\Omega \Sigma X^{\wedge n}; \H_q(\fib(\Hopf_n))) \cong \H_p(\Omega \Sigma X^{\wedge n})\otimes \H_q(\fib(\Hopf_n)) \period
    \end{equation*}
    Since 
    \begin{equation*}
        \H_p(\Omega \Sigma X^{\wedge n}) \cong \bigoplus_{i\geq 0} \wt{\H}_p(X^{\wedge n})^{\otimes i} \comma
    \end{equation*}
    and $\wt{\H}_p(X^{\wedge n})^{\otimes i}$ becomes nontrivial in degree $ i(nk+n+2) $, we find that 
    \begin{equation*}
        \H_p(\Omega \Sigma X^{\wedge n}) \cong \H_p(X^{\wedge n}) \text{\qquad for\qquad}  p < 2(nk+n+2) \period
    \end{equation*}
    In particular, \smash{$ \E^2_{p,0} = \H_p(X^{\wedge n}) $} for $ p < 2(nk+n+2) $.
    Consequently, the Serre spectral sequence has no nontrivial differentials off bidegrees
    $ (p,0) $ with $ p < 2(nk+n+2) $.

    The $\E^2$-page of this spectral sequence is very simple if $ p+q <(n+1)(k+1) $:
    in this range, $\E^2_{p,q}$ vanishes unless one of $p$ or $q$ is zero, in which case
    \begin{equation*}
        \E^2_{p,0} = \H_p(X^{\wedge n}) \andeq \E^2_{0,q} = \H_q(\fib(\Hopf_n))
    \end{equation*}
    (note that $(n+1)(k+1) \leq 2(nk+n+2)$).
    Recall that for $ p < 2(nk+n+2) $, the Serre spectral sequence has no nontrivial differentials off bidegrees $(p,0)$.
    Moreover, for $q<(n+1)(k+1)-1$, are also no nontrivial differentials with target in bidegree $(0,q)$.
    Consequently, for
    \begin{equation*}
        p+q < (n+1)(k+1) - 2 \comma
    \end{equation*}
    we find that the Serre spectral sequence collapses at the $\E^2$-page, and therefore that 
    \begin{equation*}
	   \wt{\H}_\ast(\Omega \Sigma  X) \cong \wt{\H}_{\ast}(\fib(\Hopf_n)) \oplus \wt{\H}_\ast(\Omega \Sigma X^{\wedge n}) \text{\qquad for\qquad} \ast < (n+1)(k+1) - 2 \period
    \end{equation*}
    The map $\J_{n-1}(X)\to \fib(\Hopf_n)$ then induces a homology equivalence in degrees $< (n+1)(k+1)
    - 2 $.
    We conclude by the homology Whitehead Theorem.
\end{proof}

\begin{remark}
    In the case that $ n = 2 $ and $ \dX = \Spc $, the computational proof of the metastable EHP sequence given here reduces to the proof presented in \cite[Proposition 3.1.2]{MR3654105}.
\end{remark}


\DeclareFieldFormat{labelnumberwidth}{#1}
\printbibliography[keyword=alph]
\DeclareFieldFormat{labelnumberwidth}{{#1\adddot\midsentence}}
\printbibliography[heading=none, notkeyword=alph]

\end{document}